\documentclass[11pt]{amsart}
\usepackage[utf8]{inputenc}
\usepackage{amssymb,amsthm,amsmath,amsfonts}
\usepackage[margin=1.25in]{geometry}
\usepackage{xcolor}
\usepackage[colorlinks=true,
linkcolor=purple,
filecolor=purple,
citecolor=purple,
urlcolor=blue]{hyperref}
\usepackage{graphicx}
\usepackage{tikz}
\usepackage{ifthen}
\usetikzlibrary{calc,math,fadings}
\usepackage{stmaryrd}
\usepackage{multicol}

\DeclareMathOperator{\lk}{lk}

\newtheorem*{theorem*}{Theorem}
\newtheorem{theorem}{Theorem}
\newtheorem*{claim*}{Claim}

\newtheorem*{lemma*}{Lemma}
\newtheorem{lemma}[theorem]{Lemma}
\newtheorem*{corollary*}{Corollary}
\newtheorem{corollary}[theorem]{Corollary}
\newtheorem*{proposition*}{Proposition}
\newtheorem{proposition}[theorem]{Proposition}
\newtheorem*{question*}{Question}
\newtheorem{question}[theorem]{Question}
\theoremstyle{definition}
\newtheorem*{remark*}{Remark}
\newtheorem{remark}[theorem]{Remark}
\newtheorem*{note*}{Note}
\newtheorem*{definition*}{Definition}
\newtheorem{definition}[theorem]{Definition}
\newtheorem*{example*}{Example}
\newtheorem{example}[theorem]{Example}

\title{Random Colorings in Manifolds}
\author{Chaim Even-Zohar \and Joel Hass }

\address{Chaim Even Zohar, Technion, Israel, \url{chaime@technion.ac.il}}
\address{Joel Hass, UC Davis, USA, \url{hass@math.ucdavis.edu}}

\thanks{This work was carried out in part while the first author was at the Alan Turing Institute, supported by Lloyds Register Foundation’s programme on Data-Centric Engineering. This work was carried out while the second author was visiting the Oxford Mathematical Institute and ISTA, and was a Christensen Fellow at St.~Catherine's College. The second author was partially supported by  NSF grant DMS:FRG 1760485 and BSF grant 2018313. Most of the computational work was performed with the facilities of the School of Computer Science and Engineering at the Hebrew University of Jerusalem.
}

\begin{document}

\maketitle

\begin{abstract} 
We develop a general method for constructing random manifolds and submanifolds in arbitrary dimensions. The method is based on associating colors to the vertices of a triangulated manifold, as in recent work for curves in 3-dimensional space by Sheffield and Yadin (2014). We determine conditions on which submanifolds can arise, in terms of Stiefel--Whitney classes and other properties. We then consider the random submanifolds that arise from randomly coloring the vertices. Since this model generates~submanifolds, it allows for studying properties and using tools that are not available in processes that produce general random subcomplexes. The case of 3 colors in a triangulated 3-ball gives rise to random knots and links. In this setting, we answer a question raised by de Crouy-Chanel and Simon (2019), showing that the probability of generating an unknot decays exponentially. In the general case of $k$ colors in $d$-dimensional manifolds, we investigate the random submanifolds of different codimensions, as the number of vertices in the triangulation grows. We compute the expected Euler characteristic, and discuss relations to homological percolation and other topological properties. Finally, we explore a method to search for solutions to topological problems by generating random submanifolds. We describe computer experiments that search for a low-genus surface in the 4-dimensional ball whose boundary is a given knot in the 3-dimensional sphere. 
\end{abstract}

\section{Introduction}
\label{intro}

The construction of random manifolds and submanifolds in a range of dimensions has been studied by mathematicians, physicists and biologists in a variety of contexts, generating a diversity of ideas and perspectives on this topic. We introduce here a very general method of constructing both random manifolds and random submanifolds.  Our method specializes to generate cases such as random knots, random surfaces in the 4-sphere, random surfaces bounded by a given curve in the boundary of a 4-ball, and to many other settings.

\medskip

A common approach to constructing random manifolds is based on gluing together a collection of simple pieces, typically polygons or simplices, according to some random procedure. This has been studied particularly in dimension two, where arbitrary edge gluing produces 2-manifolds \cite{harer1986euler, brooks2004random, gamburd2006poisson, pippenger2006topological, guth2011pants, linial2011expected, chmutov2016surface, petri2017random, petri2018poisson, even2020random, shrestha2020topology}. A typical property of such methods is that the valence of the resulting triangulation grows linearly with the number of simplices, while it is often desirable to have bounded valence triangulations. In dimension three and higher, random gluing of simplices along their boundaries leads to singularities with high probability, making these constructions less useful. In three dimensions, one may glue together 3-simplices with truncated corners, yielding 3-manifolds with a boundary surface of a linearly high genus~\cite{petri2020model}.

Another approach to random manifolds uses
\emph{random walks} in the mapping class group, studied in dimension three by Dunfield and Thurston, Maher, and Rivin   \cite{dunfield2006finite,maher2010random,maher2011random,rivin2014statistics,lubotzky2016random, baik2018exponential}. The random diffeomorphisms produced by these processes are used to glue together the boundaries of two handlebodies, or the two boundary surfaces of the product of a closed surface and an interval. These processes favor 3-manifolds with fixed Heegaard genus, or manifolds that fiber over a circle with a fiber of fixed genus. 

In the study of random \emph{submanifolds}, both the topological type of a submanifold and the isotopy class of an embedding are of interest. Constructions of random submanifolds have centered on models for random knots and links in $\mathbb{R}^3$, a topic that has attracted interest in physics and biology, as well as pure mathematics, leading to a variety of different approaches~\cite[a survey]{even2017models}. Random 3-manifolds also arise in this way, through Dehn fillings of random link complements, though to the authors' knowledge this has not been much explored. Random submanifolds in dimensions greater than three do not seem to have been extensively studied. 

Also related is the study of nodal lines of random fields, the zero sets of random polynomials or other analytic functions~\cite{gayet2015expected, gayet2016universal, letendre2016expected, beffara2017percolation,  letendre2019variance, armentano2018asymptotic, wigman2019expected, taylor2016vortex}. Generally, these continuous models of random submanifolds are computationally prohibitive in practice, in comparison to the more discrete models.

In problems of \emph{percolation}, one typically explores the geometry and connectivity of random subsets of $d$-dimensional space, such as vertices and edges in the cubical lattice~$\mathbb{Z}^d$. Several works have considered  the \emph{Plaquette Model} in the cubical lattice, where a cubical complex is generated based on a random set of 2-faces \cite{aizenman1983sharp, grimmett2010plaquettes, grimmett2014percolation, duncan2021homological} or $k$-faces \cite[for example]{bobrowski2020homological}. From the topological viewpoint, a drawback of these constructions is that they  yield subcomplexes rather than submanifolds. 

A method originating in the physics literature, that does produce well-behaved one and two dimensional submanifolds in~3D, is based on randomly coloring lattice vertices using two or three colors. It was introduced and investigated experimentally by Bradley et~al.~\cite{bradley1991surfaces, bradley1992growing}, with applications in mind to the study of polymers, to modeling the formation of cosmic strings~\cite{vachaspati1984formation, scherrer1986cosmic, hindmarsh1995statistical}, and to other contexts~\cite{nahum2012universal}. Sheffield and Yadin~\cite{sheffield2014tricolor} used topological tools to analyze the emergence of large components in these random 1-manifolds in 3-space. Random knotting in these curves has been explored by de Crouy-Chanel and Simon~\cite{de2019random}, with numerical experiments and conjectures. These works inspired the general scheme developed here. 

\medskip
In this paper we give constructions for random submanifolds in a wide variety of ambient manifolds and a large range of dimensions and codimensions. Our method is based on randomly coloring the vertices of a triangulation of a manifold, and then considering the submanifolds that arise as strata of the induced Voronoi tessellation. Here is a brief description of the construction, which is given later in more detail.

\smallskip
\noindent
\textbf{Main Construction} (Definitions \ref{coloring}, \ref{voronoi}, \ref{class})
{
Let $\Sigma$ be a simplicial complex, metrized by the standard Euclidean metric on each simplex. For every point in~$\Sigma$ consider its closest vertex, or its set of closest vertices if it is equidistant from several ones. Suppose that the vertices of~$\Sigma$ are assigned colors in $\{1,\dots,k\}$. The \emph{$\mathcal{C}$~color class} of $\Sigma$, that corresponds to a subset~$\mathcal{C}$ of $m \leq k$ colors, comprises all points in $\Sigma$ whose closest vertices contain all colors in~$\mathcal{C}$.}
\smallskip

While this definition applies to simplicial complexes in general, producing random subcomplexes of their barycentric subdivisions, the focus of this paper is the ability to produce manifolds. In contrast to random gluing constructions, which often lead to non-manifolds in dimension three and higher, our approach always generates well-defined submanifolds. It produces PL-submanifolds when the ambient manifold is PL. Moreover, the valency of the resulting triangulated submanifolds is bounded, with a maximal valence determined by the combinatorics of the triangulation of the ambient manifold. We discuss this construction of submanifolds and prove basic results on their simplicial structure in Section~\ref{construction}.

When applied, for example, to a 2-dimensional triangulated closed surface colored with  2 colors, the construction gives rise to a random collection of disjoint closed loops on the surface. With 3 colors, it generates a set of points on the surface, together with three collections of curves bounded by these points, yielding a 3-valent embedded graph together with a collection of simple loops. Coloring the vertices of a 3-dimensional triangulated manifold with 2 colors gives rise to random surfaces, as do 3 colors in 4 dimensions, or $k$ colors in $d=k+1$ dimensions. In general, $k$-coloring a $d$-dimensional triangulated manifold yields a proper submanifold of dimension $d-k+1$, along with higher dimensional submanifolds with boundary. We often restrict attention to the proper submanifolds, those whose boundary lies only on the boundary of the ambient manifold. We usually avoid use of more colors than the dimension.

\smallskip
Topological constraints on the submanifolds generated by the model are imposed by the choice of ambient triangulated manifold. For example, when carried out in an orientable manifold, the construction always generates submanifolds that are also orientable. In Section~\ref{universality} we discuss the universality of this construction. Large families of submanifolds are generated from some choices for the $d$-dimensional triangulated manifold and the number of colors~$k$, but other settings impose limitations on the class of manifolds that can arise. In general, we characterize the obtainable manifolds as follows.

\smallskip
\noindent
\textbf{Theorem~\ref{null-cobordant-converse}.}
\textit{
A~closed combinatorial manifold~$N$ is obtainable from some $k$-colored closed manifold~$M$ if and only if $N$ is null-cobordant.} 
\smallskip

In a specified ambient manifold~$M$, not all null-cobordant submanifolds~$N$ of~$M$ can arise from colorings. We show that the Stiefel--Whitney classes~$w_j(M)$ impose such an obstruction.
Namely, if $N$ is obtained in a manifold~$M$ with trivial $j$th Stiefel-Whitney class $w_j(M)=0$, then $w_j(N)=0$ (Proposition~\ref{stiefelwhitney}). We also give positive results for balls and spheres in codimension one and two, where specified isotopy classes of embeddings $N \subset M$, such as knots, links and Seifert surfaces, are obtainable from some triangulations and colorings of~$M$ (Propositions~\ref{link}-\ref{seifertd}).

\smallskip
In Section~\ref{random}, we turn to examine natural probabilistic models of $m$-dimensional submanifolds in $d$-dimensional spheres, tori, and balls, for any $m = d-k+1$. These random constructions come with a subdivision parameter $n$ that may grow large, but the maximum valence of the triangulation remains bounded. We then consider random submanifolds that emerge from special cases of this construction. 

\smallskip
In Section~\ref{knots}, we focus on the previously studied case $d=k=3$, which produces random 1-dimensional submanifolds. This case leads to interesting percolation phenomena in an infinite 3-dimensional grid~\cite{sheffield2014tricolor}, and to models of random knots and links in a triangulated $n \times n \times n$ grid~\cite{de2019random}. We obtain new results in this setting, and settle an open problem raised by de Crouy-Chanel and Simon~\cite{de2019random}.  

\smallskip
\noindent
\textbf{Theorem~\ref{dfw}.}
\textit{
The probability of an unknot occurring in this model of random knots decays at an exponential rate as the subdivision parameter $n$ increases.}
\smallskip

We also rigorously establish other asymptotic properties, showing for example that the number of prime summands grows at least linearly (Proposition~\ref{composite}).

\smallskip
In Section~\ref{asymptotic}, we consider random $k$-colorings of $d$-dimensional manifolds in general, and study asymptotic properties of the $(m-1)$-codimensional submanifolds that arise from subsets of $m$ colors, as the number of simplices in the triangulation grows. Since the model generates random manifolds, one can take advantage of the special properties of manifolds that are not present in models that generate random subcomplexes.  In addition to questions of connectivity and percolation for the resulting subcomplex, one can study homology properties, making use of the special nature of the homology groups of manifolds and submanifolds. In this setting, we compute the expected Euler characteristic of a random submanifold. As the subdivision parameter $n$ grows, we obtain an asymptotic density per vertex that does not depend on the ambient manifold.

\smallskip
\noindent
\textbf{Theorem~\ref{ec}.}
\textit{The expected Euler Characteristic density of a random submanifold~$N$, corresponding to a set of $m$ colors $\mathcal{C} \subseteq \{1,\dots,k\}$, in a subdivided $d$-dimensional closed manifold~$M$, whose vertices are colored iid with probabilities $(p_1,\dots,p_k)$, is}
$$ \overline{\chi}(p_1,\dots,p_k) \;=\; (1+o(1))\; \sum_{r=m}^{d+1} \genfrac{\{}{\}}{0pt}{}{d+1}{r} \, (r-1)! \; \sum_{X \subseteq\, \mathcal{C}}(-1)^{|X|}\left(-\sum_{i \in X}p_i\right)^r $$
\smallskip

In Section~\ref{genus4}, we discuss how to combine generation of random surfaces and simplicial subdivision to search for the 4-ball genus of a curve, the smallest genus of a smooth surface in the 4-ball that spans a fixed curve. We use random colorings to approach the 4-ball genus via computer experiments. Such a project has potential implications for some of the fundamental problems in topology. For example, Manolescu and Piccirillo have recently described an approach to the four-dimensional smooth Poincare Conjecture~\cite{manolescu2021zero}. If certain knots in the 3-sphere are slice, meaning that they bound a smooth disk embedded in the 4-ball, then there exists a counterexample to this conjecture. Random colorings provide a method to search for slice disks spanning these knots, and more generally to upper bound the 4-ball genus of knots in the 3-sphere. This project has to date been run on knots with a small number of crossings. It is not yet clear whether determining the 4-ball genus experimentally is within the range of computational feasibility.

We mention another application of the Voronoi coloring to 4-manifolds. Rubinstein and Tillmann~\cite{rubinstein2020multisections} showed that a vertex 3-coloring which assigns all three colors to each 4-simplex of a 4-manifold leads to
a trisection of the manifold.

\section{Colored Manifolds}
\label{construction}

This section describes a construction of submanifolds of combinatorial manifolds that is based on vertex colorings. The construction incorporates a subdivision process allowing for increasingly fine triangulations, so we review in the discussion some standard ways to subdivide simplices, cubes, and triangulated manifolds. We then verify that the construction produces not just subcomplexes,  but rather properly embedded combinatorial submanifolds.  

\begin{definition}
\label{coloring}
Let $\Sigma$ be a simplicial complex, and let $\{C_1,C_2,\dots,C_k\}$ be a set of $k$ elements, called \emph{colors}. A~\emph{$k$-coloring} of $\Sigma$ is a map 
$$ c \;:\; \mathrm{vertices}(\Sigma) \;\;\to\;\; \{C_1,C_2,\dots,C_k\} $$
\end{definition}

We extend such a coloring to $|\Sigma|$, the geometric realization of $\Sigma$. A~simplicial complex~$\Sigma$ can be realized using \emph{barycentric coordinates} for every simplex $\sigma = v_0v_1\cdots v_d \in \Sigma$, so that the coordinates of every point represent it as a weighted sum of the vertices of~$\sigma$:
$$ |\sigma| \;=\; \left\{\left.(x_0,x_1,\dots,x_d) \in \mathbb{R}^{d+1} \;\right|\; {x_0 \geq 0,\;x_1 \geq 0, \;\dots,\; x_d \geq 0, \;x_0 + \dots + x_d = 1}\right\} \;.$$
Each  simplex is metrized using the Euclidean distance in $\mathbb{R}^{d+1}$. A coloring of each point extended from the vertex coloring is defined by assigning to each point the color, or colors, of its closest vertex. This is a multi-coloring, as points equidistant from several closest vertices may be assigned more than one color. Closest vertices correspond to a point's largest barycentric coordinates, as follows.

\begin{definition}
\label{voronoi}
Let $c$ be a $k$-coloring of a simplicial complex $\Sigma$ as above. The \emph{Voronoi coloring} of $|\Sigma|$ induced from $c$ is
$$ c \;:\; |\Sigma| \;\;\to\;\; 2^{\{C_1,C_2,\dots,C_k\}} $$
such that for every simplex $\sigma = v_0v_1\cdots v_d \in \Sigma$ and point $x = (x_0,\dots,x_d) \in |\sigma|$,
$$ c(x) \;=\; \left\{c(v_i) \;\left|\; i \in \{0,\dots,d\},\; x_i = \max\limits_{j \in \{0,\dots,d\}} x_j\right.\right\} $$
\end{definition}

This coloring is well defined, since lower-dimensional faces of $|\sigma|$ are obtained by setting some of the barycentric coordinates to zero, and the colors of the removed vertices are not relevant there. In other words, the assignment of closest vertices to a point  $x \in |\Sigma|$ is the same for all subsimplices of $\Sigma$ in which $x$ lies. 

Following standard terminology from graph coloring, subsets of the space that have a given color are called {color classes}, as in the next definition. 

\begin{definition}
\label{class}
In a $k$-colored simplicial complex $\Sigma$,
the \emph{color class} of a color $C_i$ is the set of points of $|\Sigma|$ that are assigned $C_i$ by the Voronoi coloring. Note that color classes may overlap at multicolored points. The \emph{$m$-color class} corresponding to the colors $\{C'_1,\dots,C'_m\}$ contains the points of $|\Sigma|$ having all these $m$ colors, and possibly others as well. The unique $k$-color class is also called the \emph{all-color class}.
\end{definition}

Consider a point that lies in an $m$-color class corresponding to $\{C'_1,\dots,C'_m\}$, and the set of vertices that are closest to that point. This set contains at least one vertex of each color in~$\{C'_1,\dots,C'_m\}$. In particular the all-color class is the set of points that are equidistant from the $k$ vertex-sets of the $k$ assigned colors. Some examples are shown in Figure~\ref{colorclasses}.

\begin{figure}
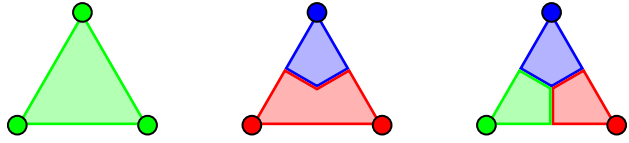

\centering
\vspace{0.5em}
\tikz{
\path[preaction={clip, postaction={draw=green, line width=2,fill=green!30}}] (330:1) -- (90:1) -- (210:1) -- cycle;
\node at (330:1)[circle,fill=green,inner sep=2.5pt,draw=black,line width=0.75]{};
\node at (90:1)[circle,fill=green,inner sep=2.5pt,draw=black,line width=0.75]{};
\node at (210:1)[circle,fill=green,inner sep=2.5pt,draw=black,line width=0.75]{};
}
\;\;\;\;\;\;\;\;
\tikz{
\path[preaction={clip, postaction={draw=blue, line width=2,fill=blue!30}}] (0,0) -- (30:0.5) -- (90:1) -- (150:0.5) -- cycle;
\node at (90:1)[circle,fill=blue,inner sep=2.5pt,draw=black,line width=0.75]{};
\path[preaction={clip, postaction={draw=red, line width=2,fill=red!30}}] (0,0) -- (30:0.5) -- (330:1) -- (210:1) -- (150:0.5) -- cycle;
\node at (210:1)[circle,fill=red,inner sep=2.5pt,draw=black,line width=0.75]{};
\node at (330:1)[circle,fill=red,inner sep=2.5pt,draw=black,line width=0.75]{};
}
\;\;\;\;\;\;\;\;
\tikz{
\path[preaction={clip, postaction={draw=blue, line width=2,fill=blue!30}}] (0,0) -- (30:0.5) -- (90:1) -- (150:0.5) -- cycle;
\node at (90:1)[circle,fill=blue,inner sep=2.5pt,draw=black,line width=0.75]{};
\path[preaction={clip, postaction={draw=red, line width=2,fill=red!30}}] (0,0) -- (30:0.5) -- (330:1) -- (270:0.5) -- cycle;
\node at (330:1)[circle,fill=red,inner sep=2.5pt,draw=black,line width=0.75]{};
\path[preaction={clip, postaction={draw=green, line width=2,fill=green!30}}] (0,0) -- (150:0.5) -- (210:1) -- (270:0.5) -- cycle;
\node at (210:1)[circle,fill=green,inner sep=2.5pt,draw=black,line width=0.75]{};
}

\vspace{0.5em}
\caption{Examples of 2-simplices colored with one, two and three colors. The 2-color classes are the arcs where two colors meet, and the center point of the rightmost triangle is a 3-color class. Note that the center point always has the colors of all vertices in the simplex.
}
\label{colorclasses}
\end{figure}

The partition of a simplicial complex into color classes admits discrete descriptions in terms of a cubical complex, or a simplicial complex in a refined triangulation. To explain these structures, we first review two subdivision processes, that convert between simplicial and cubical complexes. 

\subsubsection*{Cubing a Simplex}
Let $\sigma$ be a $d$-dimensional simplex, and suppose its $d+1$ vertices have different colors $C_0,C_1,\dots,C_d$. The color class of $C_i$ is the set of points where $x_i \geq x_j$ for all~$j$. Project this subset of the hyperplane $x_0+\dots+x_d=1$ to the hyperplane $x_i=1$, via scalar multiplication by~$1/x_i$. The homeomorphic image is a $d$-dimensional cube, as all the other $d$ coordinates range from $0$ to~$1$. 

If we restrict this projection to an $m$-color class that contains $C_i$ then the image is a $(d+1-m)$-dimensional face of the cube, with additional coordinates equal to ~$1$. Therefore, the $d+1$ cubes corresponding to the different colors consistently glue together along cubical faces, and the simplex~$\sigma$ admits a subdivision into a \emph{cubical complex}. Moreover, each face of $\sigma$ admits a similar cubing, consistent with the cubing of~$\sigma$.

\subsubsection*{Triangulating a Cube}
It is also possible to reverse this process, and subdivide a cube into simplices. The \emph{braid triangulation}~\cite{stanley2004introduction} of a $d$-dimensional cube into $d!$~simplices is based on the order types of the $d$ coordinates:
$$ [0,1]^d \;=\; \bigcup_{\pi \in S_d} \left\{ (x_1,\dots,x_d) \;|\; 0 \leq x_{\pi(1)} \leq x_{\pi(2)} \leq \cdots \leq x_{\pi(d)} \leq 1 \right\} $$
Each set in the union is a closed $d$-dimensional simplex. Its $d+1$ vertices are those vertices of the cube $\{0,1\}^d$ that satisfy the corresponding inequalities. For example, the square $[0,1]^2$ is divided into two triangles: $\boxslash\,$, corresponding to $x \leq y$ and $y \leq x$. See Figure~\ref{braid} for the braid triangulation of the 3-cube.

\begin{figure}
\centering
\vspace{0.5em}
\includegraphics[scale=2.5]{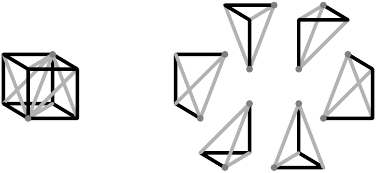}
\vspace{0.5em}
\caption{The braid triangulation of the cube~$[0,1]^3$. The 3-simplices correspond to all $3!=6$ orderings $x_i \leq x_j \leq x_k$ of the coordinates $x_1,x_2,x_3$. They all meet at the diagonal edge $x_1=x_2=x_3$, and some pairs intersect at 2-simplices on the planes $x_i = x_j$. On cubical 2-faces, they induce $\boxslash$, the braid triangulation of the 2-cube.  
}
\label{braid}
\end{figure}

\medskip

\subsubsection*{Barycentric Subdivision}
These two procedures can be combined, first partitioning a simplex into cubes, and  then triangulating the resulting cubes with simplices. Note that the cubes from the first subdivision are glued to preserve the orientation of each coordinate in $[0,1]$, so that the subdivision agrees on common low-dimensional faces. The combination produces the well-known \emph{barycentric subdivision} of a simplex into $(d+1)!$ smaller simplices.

Thus, after one barycentric subdivision of the simplicial complex $\Sigma$, the color classes have a discrete description as unions of simplices, which is not immediate from Definition~\ref{voronoi}. Each $d$-dimensional simplex in the barycentric subdivision is assigned the color of the unique original vertex it contains. Taking intersections of several colors gives the following description for color classes.

\begin{lemma}
\label{subcomplex}
Given a colored simplicial complex $\Sigma$, with barycentric subdivision $\Sigma'$, each $m$-color class $X$ is  a subcomplex of $\Sigma'$. The faces of $X$ are those simplices in $\Sigma'$ whose vertices are barycenters of original faces of $\Sigma$ whose vertices include each color defining $X$.
\end{lemma}

We now apply the Voronoi coloring to triangulated manifolds, rather than general simplicial complexes. Recall that the \emph{link} of $\sigma$ in $\Sigma$ is $\lk_\Sigma(\sigma)=\{\tau \;|\; \tau \cup \sigma \in \Sigma, \; \tau \cap \sigma = \varnothing\}$. A~\emph{combinatorial} or \emph{PL manifold} $M$ of dimension $d$ is a $d$-dimensional simplicial complex with the property that the link of every vertex is a $(d - 1)$-dimensional combinatorial sphere, meaning a combinatorial manifold PL-homeomorphic to the boundary of the $d$-dimensional simplex. In a combinatorial manifold, the link of every simplex of codimension~$r$ is a combinatorial~$S^{r-1}$. In a $d$-dimensional \emph{combinatorial manifold $M$ with boundary}, the link of every vertex is either a combinatorial $(d-1)$-sphere or a combinatorial $(d-1)$-disk, the latter meaning PL-homeomorphic to a $(d-1)$-dimensional simplex. The link of every simplex of codimension~$m$ is similarly a combinatorial $S^{m-1}$ or~$D^{m-1}$. The boundary $\partial M$ is a $(d-1)$-dimensional combinatorial manifold without boundary, that contains all simplices whose link is a disk.

If we color the vertices of a combinatorial manifold $M$ with $k$ colors as in Definition~\ref{coloring}, for some $2 \le k \le d$, then this coloring induces a Voronoi multi-coloring of $M$ as described in Definition~\ref{voronoi}. This defines color classes and $m$-color classes in $M$, which coincide with subcomplexes of its barycentric subdivision, as described in the foregoing discussion. Our purpose in the rest of this section is to show that these are submanifolds of $M$. 

We first consider the $k$-color class~$X$, consisting of points where all $k$ Voronoi regions meet. We will show that $X$ is a proper submanifold of~$M$, meaning that $X$ is a manifold, and if it has non-empty boundary then  $\partial X$ is a submanifold of~$\partial M$. This result depends essentially on the fact that we use \emph{simplices} and barycentric coordinates for the construction of Voronoi regions. The example in Figure~\ref{figure} shows how the 2-color class is not a 1-manifold in a Voronoi coloring of a square.

\begin{figure}
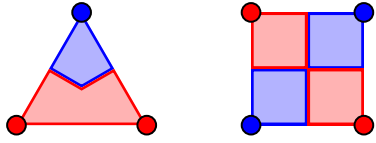

\centering
\vspace{0.5em}
\tikz{
\path[preaction={clip, postaction={draw=blue, line width=2,fill=blue!30}}] (0,0) -- (30:0.5) -- (90:1) -- (150:0.5) -- cycle;
\node at (90:1)[circle,fill=blue,inner sep=2.5pt,draw=black,line width=0.75]{};
\path[preaction={clip, postaction={draw=red, line width=2,fill=red!30}}] (0,0) -- (30:0.5) -- (330:1) -- (210:1) -- (150:0.5) -- cycle;
\node at (210:1)[circle,fill=red,inner sep=2.5pt,draw=black,line width=0.75]{};
\node at (330:1)[circle,fill=red,inner sep=2.5pt,draw=black,line width=0.75]{};
}
\;\;\;\;\;\;\;\;
\tikz{
\path[preaction={clip, postaction={draw=blue, line width=2,fill=blue!30}}] (0,0) -- (-0.75,0) -- (-0.75,-0.75) -- (0,-0.75) -- cycle;
\node at (-0.75,-0.75)[circle,fill=blue,inner sep=2.5pt,draw=black,line width=0.75]{};
\path[preaction={clip, postaction={draw=blue, line width=2,fill=blue!30}}] (0,0) -- (0.75,0) -- (0.75,0.75) -- (0,0.75) -- cycle;
\node at (0.75,0.75)[circle,fill=blue,inner sep=2.5pt,draw=black,line width=0.75]{};
\path[preaction={clip, postaction={draw=red, line width=2,fill=red!30}}] (0,0) -- (0.75,0) -- (0.75,-0.75) -- (0,-0.75) -- cycle;
\node at (0.75,-0.75)[circle,fill=red,inner sep=2.5pt,draw=black,line width=0.75]{};
\path[preaction={clip, postaction={draw=red, line width=2,fill=red!30}}] (0,0) -- (-0.75,0) -- (-0.75,0.75) -- (0,0.75) -- cycle;
\node at (-0.75,0.75)[circle,fill=red,inner sep=2.5pt,draw=black,line width=0.75]{};
}
\vspace{0.5em}
\caption{A 2-coloring of a 2-simplex gives a 2-color class that is a 1-manifold. The intersection of the blue and red regions in a square is a 1-complex that is not a 1-manifold at the central vertex, where four blue-red arcs meet at a point.
}
\label{figure}
\end{figure}

\begin{lemma} \label{manifold_in_simplex}
Let $X$ be the $k$-color class inside a $d$-simplex~$\sigma$ that is colored with $k$ distinct colors. Then $X$ is PL-homeomorphic to a $(d-k+1)$-dimensional disk in $\sigma$, with boundary a $(d-k)$-dimensional sphere contained in $\partial \sigma$.
\end{lemma}

\begin{proof}
The structure of $X$ is revealed by considering another multicoloring of the the simplex~$\sigma$, which will be shown to be equivalent to the Voronoi coloring. 

Label the vertices of $\sigma$ by $\{v_0, v_1, \dots, v_d\}$. For each of the $k$ colors $C_1,C_2,\dots,C_k$ we consider the vertices assigned that color, and define the index set of vertices colored $C_j$ to be $I_j = \{i : c(v_i)=C_j\}$. By assumption, these sets are nonempty, and their disjoint union is $\{0,\dots,d\}$. Using this partition, we define a linear map from the $d$-dimensional simplex $\sigma$ to $\mathbb{R}^k$, given in terms of the barycentric coordinates as follows.
$$ \phi(x_0,x_1,\dots,x_d) \;=\; \textstyle \left(\sum\nolimits_{i\in I_1}x_i\,,\,\sum\nolimits_{i\in I_2}x_i\,,\,\dots\,,\,\sum\nolimits_{i\in I_k}x_i \right) $$
The image of $\sigma$ under $\phi$ is exactly the $(k-1)$-dimensional simplex, also given by barycentric coordinates. We color its $k$ vertices by the distinct colors $(C_1,\dots,C_k)$. The \emph{linear} coloring $c^{\star}$ of~$\sigma$ is the pullback under~$\phi$ from the Voronoi coloring of the $(k-1)$-simplex. Namely,
$$ c^{\star}(x_0,x_1,\dots,x_d) \;=\; \textstyle \left\{C_j \;\left|\; \sum\nolimits_{i \in I_j} x_i = \max\limits_{l \in \{1,\dots,k\}} \sum\nolimits_{i \in I_l} x_i\right.\right\} $$
Note that the linear coloring gives a consistent coloring when restricted to faces of $\sigma$.  Figure~\ref{comparecolorings} shows the two colorings $c$ and $c^{\star}$ in the case $d=k=2$. 

\begin{figure}[htbp]
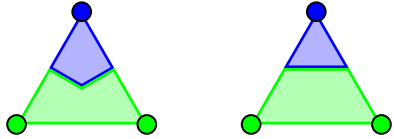

\centering
\vspace{0.5em}
\tikz{
\path[preaction={clip, postaction={draw=blue, line width=2,fill=blue!30}}] (0,0) -- (30:0.5) -- (90:1) -- (150:0.5) -- cycle;
\node at (90:1)[circle,fill=blue,inner sep=2.5pt,draw=black,line width=0.75]{};
\path[preaction={clip, postaction={draw=green, line width=2,fill=green!30}}] (0,0) -- (30:0.5) -- (330:1) -- (210:1) -- (150:0.5) -- cycle;
\node at (210:1)[circle,fill=green,inner sep=2.5pt,draw=black,line width=0.75]{};
\node at (330:1)[circle,fill=green,inner sep=2.5pt,draw=black,line width=0.75]{};
}
\;\;\;\;\;\;\;\;
\tikz{
\path[preaction={clip, postaction={draw=blue, line width=2,fill=blue!30}}] (30:0.5) -- (90:1) -- (150:0.5) -- cycle;
\node at (90:1)[circle,fill=blue,inner sep=2.5pt,draw=black,line width=0.75]{};
\path[preaction={clip, postaction={draw=green, line width=2,fill=green!30}}] (30:0.5) -- (330:1) -- (210:1) -- (150:0.5) -- cycle;
\node at (210:1)[circle,fill=green,inner sep=2.5pt,draw=black,line width=0.75]{};
\node at (330:1)[circle,fill=green,inner sep=2.5pt,draw=black,line width=0.75]{};
}

\vspace{0.5em}
\caption{On the left the Voronoi 2-coloring of a 2-simplex. On the right, its linear 2-coloring~$c^{\star}$.
}
\label{comparecolorings}
\end{figure}

Topologically the two colorings are equivalent. We show this by defining a PL homeomorphism $h$ from $\sigma$ to itself which is isotopic to the identity and takes one coloring to the other. The map $h$ is constructed successively on each skeleton of~$\sigma$. In a monochromatic simplex, where $k=1$, let $h$ be the identity map. Otherwise, $h$ is already defined on every face of $\sigma$ and we extend it to the interior. The coloring $c$ assigns all $k$ colors to the interior point $p = (\tfrac{1}{d+1}, \dots, \tfrac{1}{d+1})$. For $i \in \{0,\dots,d\}$ let $n_i$ be the number of vertices colored with $c(v_i)$, so $n_i = |I_j|$ where $c(v_i) = C_j$. Define the map $h$ to take $p$ to the point
$$ h(p) \;=\; \left(\frac{1}{k\,n_0}, \,\frac{1}{k\,n_1}, \,\dots,\, \frac{1}{k\,n_d} \right) \;.$$
Note that $h(p)$ is assigned all $k$ colors by the linear coloring~$c^{\star}$, because the barycentric coordinates of each set $I_j$ sum to~$1/k$. Consider a boundary point $q \in \partial \sigma$. All the points within the straight segment between $p$ and~$q$ have the same Voronoi coloring~$c(q)$, since a convex combination with $p$ preserves the order relations between the coordinates. Similarly, all the points within the straight segment between $h(p)$ and~$h(q)$ are assigned by the linear coloring~$c^{\star}$ the same set of colors $c^{\star}(h(q)) = c(q)$. We therefore extend the PL map~$h$ linearly to all interior points of $\sigma$, and then we have $c^{\star} \circ h = c$ on the whole simplex. Since the point $p$ is the barycenter of~$\sigma$, the map $h$ is linear on the simplices of the barycentric subdivision, and yields a combinatorially equivalent subdivision of $\sigma$ into simplices. The color classes of~$c^{\star}$ are given by the same  subcomplexes as those of $c$ in the barycentric subdivision.

Having established the equivalence of these two triangulations, we show that $h(X)$ is a properly embedded disk in~$\sigma$. Under the coloring~$c^{\star}$, the $k$-color class $h(X)$ is given by $\phi^{-1}(1/k,\dots,1/k)$, which is defined by the linear equations  
$$
\sum\nolimits_{i \in I_1} x_i \;=\; \sum\nolimits_{i \in I_2} x_i \;=\; \dots \;=\; \sum\nolimits_{i \in I_k} x_i \;=\; \frac1k \;.
$$
The last equality is redundant because the sum of all barycentric coordinates is~$1$.
Hence the set $h(X)$ is an intersection of $(k-1)$ independent hyperplanes with the $d$-simplex, which is a convex set. This intersection contains the interior point~$h(p)$, so it is a properly embedded $(d-k+1)$-disk, and its intersection with $\partial \sigma$ is a $(d-k)$-sphere. By the homeomorphism~$h$, the same holds for the $k$-color class~$X$, as stated in the lemma. 
\end{proof}

\begin{lemma} \label{manifold_in_M}
Let $X$ be the $k$-color class inside a combinatorial $d$-manifold $M$, colored using $k$ colors, for $2 \leq k \leq d$. Then $X$ is a union of simplices that form a proper combinatorial $(d-k+1)$-dimensional submanifold of~$M$. The boundary~$\partial X$ is the submanifold~$X \cap \partial M$.
\end{lemma}

\begin{proof}
Let $M'$ be the combinatorial $d$-manifold obtained from $M$ by barycentric subdivision. The vertices of $M'$ correspond to barycenters of faces in~$M$, of all dimensions $0,\dots,d$, and the faces of~$M'$ correspond to increasing chains of faces in~$M$, ordered by inclusion. By Lemma~\ref{subcomplex}, the $k$-color class $X$ is a subcomplex of~$M'$, and by Lemma~\ref{manifold_in_simplex} it is a union of $(d-k+1)$-dimensional closed simplices. Moreover, $X$ is induced from~$M'$ by restriction to a subset of its vertices, the barycenters of faces in~$M$ having vertices of each color. 

To show that $X$ is a combinatorial manifold, we check the links of its vertices. Let $\sigma \in M$ be a face whose vertices are together assigned all $k$ colors. Its barycenter $v_\sigma$ is a vertex of the subcomplex~$X \subset M'$. We first consider the case where $M$ has no boundary.

If $\dim \sigma=d$, then the link $\lk_{M'}(v_\sigma)$ is contained in~$\sigma$, because its neighbors are $v_\tau$ for all subfaces $\tau \subset \sigma$. Restricting to $k$-colored faces in~$M'$, also $\lk_X(v_\sigma)$ is contained in~$\sigma$. By Lemma~\ref{manifold_in_simplex}, $v_\sigma$ is an internal vertex of the combinatorial $(d-k+1)$-disk $\sigma \cap X$. Hence $\lk_X(v_\sigma)$ is a combinatorial $(d-k)$-sphere. 

Consider $\sigma \in M$ with exactly $k$ vertices of different colors. Since $M$ is combinatorial, $\lk_M(\sigma)$ is a combinatorial $(d-k)$-sphere. It includes all faces $\tau \in M$ disjoint from~$\sigma$ such that $\sigma \cup \tau \in M$. Its barycentric subdivision is the restriction of~$M'$ to the vertices~$v_\tau$. Also $\lk_X(v_\sigma)$ contains the vertices $v_{\sigma \cup \tau}$ for exactly the same set of faces~$\tau$ in~$M$, as these are the $k$-colored neighbors of~$v_\sigma$. This bijection naturally extends to higher dimensional simplices in the two links. Thus $\lk_X(v_\sigma)$ is isomorphic to the subdivided $(d-k)$-sphere~$\lk_M(\sigma)$.

Interpolating between these two cases, suppose $\dim \sigma = m$ for some $k-1 < m < d$. By Lemma~\ref{manifold_in_simplex}, $\lk_X(v_\sigma)$ contains a combinatorial $(m-k)$-sphere, with vertices $v_\tau$ for faces $\tau \subset \sigma$ that have vertices of each color. An embedding of $\lk_M(\sigma)$ in $\lk_X(v_\sigma)$ as above shows that $\lk_X(v_\sigma)$ also contains a disjoint combinatorial $(d-m-1)$-sphere, with vertices $v_\tau$ for all faces $\tau \supset \sigma$. The faces of $\lk_X(v_\sigma)$ are increasing chains of vertices from the two spheres. These are all unions of two faces from the two simplicial complexes, which is their \emph{join}, giving another combinatorial sphere $S^{m-k} \ast S^{d-m-1} = S^{d-k}$.

In conclusion, since all links of the vertices of $X$ are combinatorial $(d-k)$-spheres, $X$~is a combinatorial $(d-k+1)$-manifold without boundary.

For a boundary face $\sigma \in \partial M$ with vertices together assigned all $k$ colors, $\lk_X(v_\sigma)$ is a combinatorial $(d-k)$-disk, since $\lk_M(\sigma)$ is a $(d-m-1)$-disk and $S^{m-k} \ast D^{d-m-1} = D^{d-k}$. Therefore $X$ is a combinatorial $(d-k+1)$-manifold with boundary, and its vertices satisfy $v_\sigma \in \partial X$ if and only if $\sigma \in \partial M$. It follows that both $\partial X$ and~$X \cap \partial M$ are combinatorial $(d-k)$-manifolds without boundary, the former as the boundary of~$X$ and the latter as the $k$-color class of~$\partial M$. Moreover, it follows from the barycentric subdivision construction that $X \cap \partial M$ contains the same vertices of~$M'$ as~$\partial X$, and all faces spanned by these vertices. Therefore, $\partial X = X \cap \partial M$. 
\end{proof}

We also want to consider color classes that use fewer colors than the maximal possible number~$k$. We show that these form submanifolds with boundary. 

\begin{lemma}
\label{submanifold}
Let $X$ be an $m$-color class inside a combinatorial $d$-manifold $M$ colored using $k$ colors, where $1 \leq m \leq k \leq d$. Then $X$ is a union of  simplices that form a combinatorial $(d-m+1)$-dimensional submanifold of~$M$.  Furthermore, $\partial X = (X \cap \partial M) \cup X^+$, where $X^+$ is the subset of~$X$ colored with more than $m$ colors.
\end{lemma}

\begin{proof}
The special case $m=k$ was treated in Lemma~\ref{manifold_in_M}. Without loss of generality, we can examine an $m$-color class of $M$ that corresponds to the set $\{C_1,C_2,\dots,C_m\}$ of the first $m<k$ colors. If not all these colors are represented in the vertices of some $d$-dimensional simplex~$\sigma$, then $X$ is disjoint from~$\sigma$. If exactly $C_1,\dots,C_m$ appear in~$\sigma$, then $X \cap \sigma$ is a $(d-m+1)$-disk and $X \cap \partial\sigma$ is a $(d-m)$-sphere by Lemma~\ref{manifold_in_simplex}. 

Consider a $d$-dimensional simplex $\sigma \in M$ where strictly more colors than $C_1,\dots,C_m$ appear. Using the construction in the proof of Lemma~\ref{manifold_in_simplex}, the set $X \cap \sigma$ is described by the solutions of $m-1$ independent linear equations
$$
\sum\nolimits_{i \in I_1} x_i \;=\; \sum\nolimits_{i \in I_2} x_i \;=\; \dots \;=\; \sum\nolimits_{i \in I_m} x_i 
$$
and $k-m$ linear inequalities
$$
\sum\nolimits_{i \in I_j} x_i \;\leq\; \sum\nolimits_{i \in I_1} x_i \;\;\;\;\;\;\;\; j \in \{m+1,\dots,k\}\;.
$$
This intersection of convex sets is a $(d-m+1)$-dimensional disk embedded in~$\sigma$. Since at least one of the last $k-m$ colors appears in~$\sigma$, the set of inequalities defining this set is not empty. Besides points of~$\partial \sigma$, the boundary of this disk contains interior points of~$\sigma$ where inequalities become equality, one of which is the translated barycenter $h(p) \in \sigma$. Such boundary points occur in $(m+1)$-color classes.

Therefore, $X$ is a union of $(d-m+1)$-dimensional disks, made of closed simplices in the subdivision of $M$. As in Lemma~\ref{manifold_in_M}, to show that it is a combinatorial submanifold we inspect the link of vertices $v_\sigma \in X$ for $\sigma \in M$. If $\dim \sigma = d$, then $\lk_X(v_\sigma)$ is a $(d-m)$-sphere or a $(d-m)$-disk, depending on whether $\sigma$ has more than $m$ colors as shown above. If $\dim \sigma = m-1$ then the isomorphism with $\lk_M(\sigma)$ described in the proof of Lemma~\ref{manifold_in_M} shows that $\lk_X(v_\sigma)$ is either a $(d-m)$-sphere or a $(d-m)$-disk, depending on whether~$\sigma \in \partial M$. In the intermediate cases $m-1 < \dim \sigma < d$, we similarly obtain $(d-m)$-dimensional spheres and disks from joins of spheres and disks as above. Hence $X$ is a combinatorial submanifold with boundary, whose triangulation is a subcomplex of the barycentric subdivision of~$M$.

The vertex $v_\sigma$ is a boundary vertex of~$X$ if either $\sigma \in \partial M$ has vertices of the first $m$ colors or $\sigma \in M$ has additional colors. The set $X^+$ is a proper submanifold of~$M$ with boundary $X^+ \cap \partial M$ by Lemma~\ref{manifold_in_M}, being an $(m+1)$-color class if one assigns a single new color to the last $k-m$ colors. Also $X \cap \partial M$ is a submanifold of $\partial M$ with boundary $X^+ \cap \partial M$, by induction on the dimension. Hence $(X \cap \partial M) \cup X^+$ is a $(d-m)$-submanifold with no boundary, and by the same reasoning as in the previous lemma it is equal to $\partial X$. 
\end{proof}

\begin{remark}
An $m$-color class may not be a proper submanifold, since it can have a boundary in the interior of $M$, but several of these can be glued together to eliminate  boundary components. This can be done by recoloring the $k-m$ remaining color classes using some assignment from the first $m$ colors, and applying Lemma~\ref{manifold_in_M} with $k=m$. For example, in a 4-coloring by $\{$R,G,B,W$\}$, the 2-color class RG has a boundary, but the union of RG, WG, RB, and WB is a proper submanifold of codimension~1, since it is the interface between R$\cup$W and G$\cup$B. 
\end{remark}

In dimensions above four, Lemmas \ref{manifold_in_M} and~\ref{submanifold} require the condition that the ambient manifold $M$ is combinatorial. This is illustrated  by taking $M$ to be the non-combinatorial manifold formed by taking the double suspension of the Poincare Homology Sphere~$P$. This double suspension is a triangulated manifold that is homeomorphic but not PL-homeomorphic to the 5-sphere. Suppose we color all vertices of $M$ red except for one of the two cone-points in the second suspension, which we color blue. Then the 2-color class is PL-homeomorphic to a single suspension of~$P$. This single suspension is not a manifold, since removing one of its two suspension points leaves an end that is not simply connected at infinity.

\section{Which Manifolds Arise}
\label{universality}

Having shown that coloring the vertices of a combinatorial manifold $M$ produces a submanifold $N$, we consider the question of which manifolds arise. Note that some topological manifolds of dimension four and above do not admit combinatorial triangulations. These manifolds cannot be obtained from the Voronoi construction carried out within a combinatorial manifold, which we showed yields a combinatorial manifold. From now on we restrict attention to combinatorial manifolds.

We first examine which manifolds $N$ arise as a $k$-color class in a manifold $M$. There are several ways to ask such a question. Note for example that every manifold~$N$ is the 1-color class obtained by taking $M=N$ and assigning a single color to all vertices. For every $k \geq 1$, $N$~is the $k$-color class in $M=N \times B$, where $B$ is a $(k-1)$-dimensional simplex having $k$ colors, and the coloring of~$M$ is induced from~$B$. In this construction $M$~has a boundary even if $N$ is closed. To generate a closed ambient manifold, one can restrict the previous construction to~$N \times \partial B$, and then each of the $(k-1)$-color classes gives a copy of~$N$, while the $k$-color class is empty. Another construction of a closed manifold with a coloring giving a submanifold~$N$ is $N \times \partial B$, where $B$ is a $k$-dimensional simplex with one of the $k$ colors repeating twice. Here the all-color class is a disjoint union of two copies of~$N$. This leads to the following question.

\begin{question*} [\textbf I]
Given a closed manifold~$N$, and an integer $k \geq 2$, is there a closed $k$-colored manifold~$M$ in which $N$~forms the entire $k$-color class?
\end{question*}

We give examples showing that the answer depends on~$N$.  Example~\ref{spheres} shows that spheres occur as a $k$-color class for all~$k$, while Example~\ref{cp2} shows that the projective plane never occurs as the entire $k$-color class for any~$k \geq 2$.

In the case of $k=2$ colors, note that the $d$-sphere $N=S^d$ is an equator in $M=S^{d+1}$, so it can be obtained as the interface between the colors. A~more general construction is based on the join of spheres, as follows.

\begin{example}
\label{spheres}
Let $N = S^d$, a triangulated $d$-sphere, and~$k \geq 2$. We take $M = N \ast S^{k-2}$, the join of~$N$ with the boundary of a $(k-1)$-simplex, and color the $k$ vertices of~$S^{k-2}$ with all $k$ colors, and also color all the vertices of~$N$ with one of these colors. The all-color class is one copy of~$N$ while~$M = S^{d+k-1}$ is closed.
\end{example}

Note that we cannot use joins in general, because $M$~is required to be a manifold. Nor can we in general embed~$N$ as a separating codimension-one submanifold, since some manifolds are not boundaries, such as the following example.
 
\begin{example}
\label{cp2}
For $k \geq 2$ colors, the real projective plane $N = RP^2$ is not the entire $k$-color class in any closed $(k+1)$-dimensional manifold~$M$. 
\end{example}

This example is a special case of Lemma~\ref{null-cobordant}, which describes a general obstruction that applies to a wide family of manifolds, including $RP^n$ and $CP^n$ for~$n \geq 2$.

\begin{lemma}
\label{null-cobordant}
If a closed manifold~$M$ is colored with $k \ge 2$  colors, and a manifold~$N$ occurs as the entire $k$-color class of~$M$, then $N$ is null-cobordant. 
\end{lemma}

\begin{proof}
  Suppose that $N$ is the entire $k$-color class in a closed $k$-colored manifold~$M$, where $k \geq 2$. Lemma~\ref{submanifold} implies that if one of the colors is omitted, then $N$ is the entire boundary of the resulting $(k-1)$-color class.
\end{proof}

\begin{remark}
\label{sw-numbers}
Many manifolds are not null-cobordant. This can be determined by a family of invariants known as the Stiefel--Whitney numbers~\cite{milnor1974characteristic}. For smooth manifolds, Thom showed that a smooth manifold is a boundary if and only if all its Stiefel--Whitney numbers vanish~\cite{thom1954quelques}. This result extends to PL-manifolds~\cite{buoncristiano1991characteristic}.
\end{remark}

We now consider the converse direction, and show that null-cobordism is sufficient. For example, the Klein bottle is the boundary of the solid Klein bottle, and by doubling the latter we obtain the Klein bottle as a separating surface between two color classes. Generalizing this construction, we establish the following result, which answers Question~(I).

\begin{theorem}
\label{null-cobordant-converse}
Let $k \geq 2$. A manifold~$N$ occurs as the $k$-color class in some closed $k$-colored manifold~$M$ if and only if $N$ is null-cobordant.
\end{theorem}

Before giving a construction that proves this, we discuss its main building block. The \emph{double} of a manifold~$X$ with a nonempty boundary is defined as 
$$  DX \;=\; (X \times \{a,b\}) \;/\; \{(x,a) \sim (x,b) \text{ for } x \in \partial X\} \;.$$
The double of $X$ is a closed manifold, and has a natural involution $\tau : DX \to DX$ interchanging $(x,a)$ and $(x,b)$ and fixing $\partial X = \partial X \times \{a\} \subset DX$.

A first attempt to realize the double of a combinatorial manifold~$X$ as a simplicial complex is to duplicate all nonboundary vertices and faces. This runs into technical problems, as an~internal face whose vertices lie on the boundary gives rise to two identical simplices. For example, if $X$ consists of one edge with two boundary vertices, then doubling the edge does not give a simplicial complex. Instead of addressing this issue by subdividing such faces beforehand, we  take another approach that  lets us extend colorings of $X$ in a natural way.

The manifold $X$ contains a collar neighborhood $\partial X \times [0,\varepsilon)$ of the boundary. Therefore, $X$~is PL homeomorphic to $X' = X \cup (\partial X \times [0,1])$ where $\partial X \times \{0\}$ is glued to $\partial X \subset X$, and $\partial X' = \partial X \times \{1\}$. We triangulate the product without additional vertices, see Lemma~\ref{product} for an explicit and general definition of such a triangulation. Every simplex of~$X'$ with all vertices in the boundary is fully contained in the boundary. Hence duplicating the interior faces of~$X'$ gives a simplicial triangulation of the topological manifold~$DX$. Below we assume this PL triangulation for $DX$. Note that in this triangulation, the separating $\partial X'$ in the middle of~$DX$ has parallel simplicial copies on both sides.

The next lemma shows that a $k$-color class in the boundary of a manifold~$X$ can be made into a $(k+1)$-color class in an appropriate coloring of its double~$DX$.

\begin{lemma}
\label{double}
Let $X$ be a $k$-colored combinatorial manifold with boundary~$\partial X$, and let $Y$ be the $k$-color class of~$\partial X$. The double~$DX$ admits a combinatorial triangulation and a $(k+1)$-coloring whose $(k+1)$-color class is isotopic in $DX$ to the embedding of $Y$ in~$\partial X$. 
\end{lemma}

\begin{proof}
The double $DX$ is triangulated as in the foregoing discussion: thickening the boundary of~$X$, to obtain $X' = X \cup (\partial X \times [0,1])$, duplicating $X'$, and identifying the two copies of~$\partial X'$. The resulting closed combinatorial manifold is 
$$ DX \;=\; X \,\cup\, (\partial X \times [0,1]) \,\cup\, \tau(\partial X \times [0,1]) \,\cup\, \tau(X) $$
with consecutive terms glued along three parallel $\partial X$ copies: $x \sim (x,0)$, $(x,1) \sim \tau((x,1))$, and $\tau((x,0)) \sim \tau(x)$ for $x \in \partial X$. The given $k$-color class embeds as $Y \subset \partial X \times\{0\} \subset DX$.

The $(k+1)$-coloring of~$DX$ is defined as follows. The vertices of the first copy of~$X$ keep their given colors, say $C_1,\dots,C_k$. The vertices of $\partial X \times \{1\}$ and $\tau(X)$ are assigned a new color $C_0$. Let $W$ be the $(k+1)$-color class in $DX$. We show that $Y$ is isotopic to~$W$.

Let $U$ denote the boundary of the color class of $C_0$. By Lemma~\ref{submanifold}, $U$~is a codimension one submanifold of~$DX$, made of the 2-color classes of $\{C_0,C_i\}$ for all $1 \leq i \leq k$. Note that $W \subset U \subset \partial X \times (0,1)$, since no faces in the rest of~$DX$ have vertices of $C_0$ and another color. The proof goes by an isotopy from $\partial X \times\{0\}$ to~$U$, which moves $m$-color classes of~$\partial X$ to $(m+1)$-color classes of~$DX$ corresponding to the same colors and~$C_0$. This isotopy takes the $k$-color class~$Y$ to~$W$.

By Lemma~\ref{subcomplex}, the color classes in~$DX$ are subcomplexes of its barycentric subdivision, with new vertices at the barycenters of the faces. Here it is convenient to move the points introduced during barycentric subdivision so that they lie on $\partial X \times \{\tfrac12\}$ for any face that intersects this hyperplane. See the proof of Lemma~\ref{manifold_in_simplex} for a similar equivalence of subdivisions. Observe that this choice implies $U = \partial X \times \{\tfrac12\}$.

We consider each prism $P_\sigma = \sigma \times [0,1]$ where $\sigma \in \partial X$, and isotope $\sigma \times \{0\}$ to $U \cap P_\sigma$. Gluing these pieces inductively on $\dim \sigma$ gives the isotopy from $\partial X \times \{0\}$ to $U$. Without loss of generality, suppose that the vertices of~$\sigma$ are assigned distinct colors. Otherwise, identify the appropriate sets of colors to obtain the correspondence between color classes.

\begin{figure}
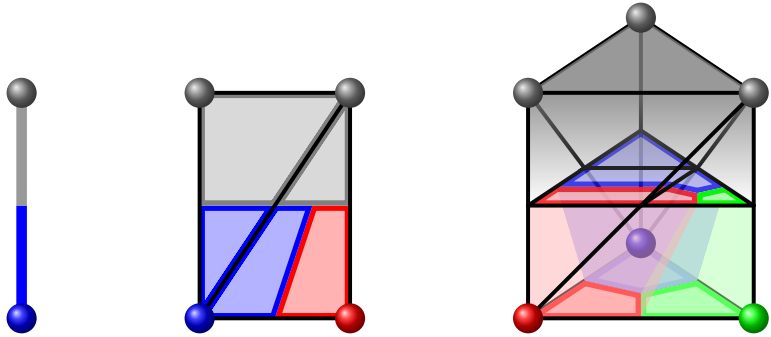

\centering
\tikz{
\path[draw=blue, line width=4] (0,0) -- (0,1.5);
\path[draw=gray!80, line width=4] (0,1.5) -- (0,3);
\shade [ball color=blue] (0,0) circle (0.2);
\shade [ball color=gray] (0,3) circle (0.2);}
\;\;\;\;\;\;\;\;\;\;\;\;\;\;\;\;
\tikz{
\path[preaction={clip, postaction={draw=red, line width=4,fill=red!30,opacity=1}}] (1,0) -- (2,0) -- (2,1.5) -- (1.5,1.5) -- cycle;
\path[preaction={clip, postaction={draw=blue, line width=4,fill=blue!30,opacity=1}}] (0,0) -- (1,0) -- (1.5,1.5) -- (1,1.5) -- (0,0) -- (1,1.5) -- (0,1.5) -- cycle;
\path[preaction={clip, postaction={draw=gray, line width=4,fill=gray!30,opacity=1}}] (0,1.5) -- (2,1.5) -- (2,3) -- (1,1.5) -- (2,3) -- (0,3) -- cycle;
\path[draw=black, line width=1.5] (0,0) -- (0,3) -- (2,3) -- (0,0) -- (2,0) -- (2,3);
\shade [ball color=blue] (0,0) circle (0.2);
\shade [ball color=red] (2,0) circle (0.2);
\shade [ball color=gray] (0,3) circle (0.2);
\shade [ball color=gray] (2,3) circle (0.2);
}
\;\;\;\;\;\;\;\;\;\;\;\;\;\;\;\;
\tikz{
\path[draw=black, line width=1.5] (1.5,4) -- (0,3) -- (1.5,1) -- (1.5,4) -- (3,3) -- (1.5,1) -- (3,0) (1.5,1) -- (0,0);
\path[preaction={clip, postaction={draw=blue, line width=4,fill=blue!30,opacity=1}}] (0.75,0.5) -- (1.5,1) -- (2.25,0.5) -- (1.5,0.333) -- cycle;
\path[preaction={clip, postaction={draw=blue, line width=4,fill=blue!30,opacity=0.8}}] (0.375,1.75) -- (1.5,2.5) -- (2.625,1.75) -- (2.25,1.666) -- (1.875,1.75) -- cycle;
\path[preaction={clip, postaction={fill=blue!30,opacity=0.8}}] (0.375,1.75) -- (1.875,1.75) -- (2.25,1.666) -- (2.625,1.75) -- (2.25,0.5) -- (1.5,0.333) -- (0.75,0.5) -- cycle;
\shade [ball color=blue] (1.5,1) circle (0.2);
\path[preaction={clip, postaction={draw=green, line width=4,fill=green!30,opacity=1}}] (1.5,0) -- (3,0) -- (2.25,0.5) -- (1.5,0.333) -- cycle;
\path[preaction={clip, postaction={draw=green, line width=4,fill=green!30,opacity=1}}] (2.25,1.5) -- (3,1.5) -- (2.625,1.75) -- (2.25,1.666) -- cycle;
\path[preaction={clip, postaction={fill=green!30,opacity=0.5}}] (2.25,1.5) -- (3,1.5) -- (3,0) -- (1.5,0) -- (1.5,0.333) -- (2.25,1.666) -- cycle;
\path[preaction={clip, postaction={draw=red, line width=4,fill=red!30,opacity=1}}] (1.5,0) -- (0,0) -- (0.75,0.5) -- (1.5,0.333) -- cycle;
\path[preaction={clip, postaction={draw=red, line width=4,fill=red!30,opacity=0.8}}] (2.25,1.5) -- (0,1.5) -- (0.375,1.75) -- (1.875,1.75) -- (2.25,1.666) -- cycle;
\path[preaction={clip, postaction={fill=red!30,opacity=0.5}}] (2.25,1.5) -- (0,1.5) -- (0,0) -- (1.5,0) -- cycle;
\path[preaction={clip},postaction={draw=black,line width=1.5,line cap=round}] (0.75,2) -- (0,1.5) (0,1.5) -- (1.5,1.5) -- (2.25,2) (2.25,2) -- (0.75,2) -- (1.5,2.5) -- (3,1.5) (3,1.5) -- (1.5,1.5);
\path[preaction={clip, postaction={draw=gray, line width=4,fill=gray,opacity=0.8}}] (0,3) -- (3,3) -- (1.5,4) -- cycle;
\tikzfading[name=myfading,
  top color=transparent!20, bottom color=transparent!100]
\fill[gray,path fading=myfading] (0,1.5) -- (0,3) -- (3,3) -- (3,1.5) -- cycle;
\path[draw=black, line width=1.5] (0,0) -- (3,0) -- (3,3) -- (0,0) -- (0,3) -- (3,3);
\shade [ball color=gray] (1.5,4) circle (0.2);
\shade [ball color=red] (0,0) circle (0.2);
\shade [ball color=green] (3,0) circle (0.2);
\shade [ball color=gray] (0,3) circle (0.2);
\shade [ball color=gray] (3,3) circle (0.2);
}
\caption{The $(d+1)$-colored triangulated prism $\sigma \times [0,1]$ in dimensions $d=1,2,3$ as in the proof of Lemma~\ref{double}. Each colored region is a union of smaller simplices obtained by barycentric subdivision, but this is not shown. An isotopy takes each $m$-color class at height~0 to an $(m+1)$-color class at height~$\tfrac12$.}
\label{prisms}
\end{figure} 

The prism $P_\sigma$ is triangulated without additional vertices, see Lemma~\ref{product} below, using the vertices of $\sigma \times \{0\}$ and $\sigma \times \{1\}$, with $d = \dim\sigma+1$ simplices of the top dimension~$d$. If $\sigma$'s vertices are colored by $C_1,\dots,C_d$ then the vertex colorings of these $d$ simplices are as follows:
$$ (C_0,C_1,\dots,C_d),\; (C_0,C_0,C_2,\dots,C_d),\; (C_0,C_0,C_0,C_3,\dots,C_d),\; \dots,\; (C_0,\dots,C_0,C_d) $$
Note that adjacent simplices in this sequence share a $(d-1)$-dimensional face. Illustrations are provided in Figure~\ref{prisms} for $d=1,2,3$. The intersection $U \cap P_\sigma = \sigma \times \{\tfrac12\}$ is a union of $(d-1)$-dimensional topological disks in  these $d$ simplices. Similarly, each $(m+1)$-color class that corresponds to $C_0$ and $m$ additional colors is a union of $(d-m)$-dimensional disks in consecutive simplices, that fit together to form one disk. In particular, the $(d+1)$-color class is a point in the first simplex. The resulting stratification of color classes in $\sigma \times \{\tfrac12\}$ is equivalent to that of the $d$-colored $\sigma = \sigma \times \{0\} \subset \partial X \times\{0\}$, and similarly to that of every $\sigma \times \{t\}$ where $t \in (0,\tfrac12)$. Hence, we can extend the given isotopy in $\partial \sigma \times [0,\tfrac12]$ to an isotopy in $\sigma \times [0,\tfrac12]$, starting with the one-point $d$-color class, then for each $(d-1)$-color class, and so~on.
\end{proof}

The following lemma states that the double of a manifold is itself a boundary. This  lets us iterate the previous lemma in the proof of Theorem~\ref{null-cobordant-converse}.

\begin{lemma}
\label{doubleisboundary}
For every combinatorial manifold~$X$, there exists a combinatorial manifold~$Z$ such that $\partial Z = DX$.
\end{lemma}

\begin{proof}
By the above discussion, we assume that $X$ is triangulated such that its double $DX = (X \cup \tau(X)) / (\partial X \equiv \tau (\partial X))$ is a combinatorial manifold. We also assume that the products $X \times [0,1]$ and $DX \times [0,1]$ are triangulated as defined generally in Lemma~\ref{product} below, letting the vertices of~$\partial X$ be smaller than the others in the arbitrary ordering of~$X$ or~$DX$. It follows that no edge or simplex connects $(x_0,0)$ to $(x_1,1)$ for $x_0 \not\in \partial X$ and~$x_1 \in \partial X$.

The manifold $Z$ with $\partial Z = DX$ is constructed from $DX \times [0,1]$ by eliminating the boundary component $DX \times \{0\}$. This is done by identifying its two halves $X \times \{0\}$ and $\tau(X) \times \{0\}$ with each other. We thus define 
\begin{align*}
Z \;&=\; ( DX \times [0,1]) \;/\; \{(x,0) \sim (\tau(x), 0) \text{ for } x \in X\} \\
\;&=\; ((X \times [0,1]) \,\cup\, (\tau(X) \times [0,1])) \;/\; \{(x,t) \sim (\tau(x), t) \text{ if } t=0 \text{ or  } x \in \partial X\} 
\end{align*}
See Figure~\ref{doubling} for an illustration of this construction. It remains to show that $Z$ is a combinatorial manifold bounded by~$DX$.

\begin{figure}
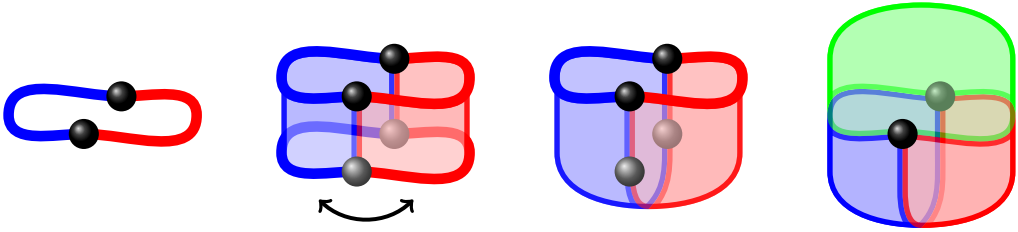

\centering
\tikz{
\clip (-1.5,-2) rectangle (2,2);
\path[draw=blue, line width=4] (0,0) to[out=180,in=270] (-1,0.25) to[out=90,in=180] (0.5,0.5);
\path[draw=red, line width=4] (0,0) to[out=0,in=270] (1.5,0.25) to[out=90,in=0] (0.5,0.5);
\shade [ball color=black] (0,0) circle (0.2);
\shade [ball color=black] (0.5,0.5) circle (0.2);}
\tikz{
\clip (-1.5,-2.5) rectangle (2,1.5);
\path[draw=blue, line width=4] (0,-1) to[out=180,in=270] (-1,-0.75) to[out=90,in=180] (0.5,-0.5);
\path[preaction={clip, postaction={draw=blue, line width=4,fill=blue!30,opacity=0.8}}] (0.5,0.5) to[out=180,in=90] (-1,0.25) -- (-1,-0.75) to[out=90,in=180] (0.5,-0.5) -- cycle;
\path[draw=red, line width=4] (0,-1) to[out=0,in=270] (1.5,-0.75) to[out=90,in=0] (0.5,-0.5);
\path[preaction={clip, postaction={draw=red, line width=4,fill=red!30,opacity=0.8}}] (0.5,0.5) to[out=0,in=90] (1.5,0.25) -- (1.5,-0.75) to[out=90,in=0] (0.5,-0.5) -- cycle;
\shade [ball color=gray] (0.5,-0.5) circle (0.2);
\path[preaction={clip, postaction={draw=blue, line width=4,fill=blue!30,opacity=0.6}}] (0,0) to[out=180,in=270] (-1,0.25) -- (-1,-0.75) to[out=270,in=180] (0,-1) -- cycle;
\path[draw=blue, line width=4] (0,0) to[out=180,in=270] (-1,0.25) to[out=90,in=180] (0.5,0.5);
\path[preaction={clip, postaction={draw=red, line width=4,fill=red!30,opacity=0.6}}] (0,0) to[out=0,in=270] (1.5,0.25) -- (1.5,-0.75) to[out=270,in=0] (0,-1) -- cycle;
\path[draw=red, line width=4] (0,0) to[out=0,in=270] (1.5,0.25) to[out=90,in=0] (0.5,0.5);
\shade [ball color=black] (0,0) circle (0.2);
\shade [ball color=gray] (0,-1) circle (0.2);
\shade [ball color=black] (0.5,0.5) circle (0.2);
\draw[black,line width=1.5,<->]
(-0.5,-1.375) to[out=-45,in=-135] (0.75,-1.375);
}
\tikz{
\clip (-1.5,-2.5) rectangle (2,1.5);
\path[preaction={clip, postaction={draw=blue, line width=4,fill=blue!30,opacity=0.8}}] (0.5,0.5) to[out=180,in=90] (-1,0.25) -- (-1,-0.75) to[out=270,in=180] (0.25,-1.5) to[out=45,in=270] (0.5,-0.5) -- cycle;
\path[preaction={clip, postaction={draw=red, line width=4,fill=red!30,opacity=0.8}}] (0.5,0.5) to[out=0,in=90] (1.5,0.25) -- (1.5,-0.75) to[out=270,in=0] (0.25,-1.5)  to[out=45,in=270] (0.5,-0.5) -- cycle;
\shade [ball color=gray] (0.5,-0.5) circle (0.2);
\path[preaction={clip, postaction={draw=blue, line width=4,fill=blue!30,opacity=0.5}}] (0,0) to[out=180,in=270] (-1,0.25) -- (-1,-0.75) to[out=270,in=180] (0.25,-1.5) to[out=180,in=270] (0,-1) -- cycle;
\path[draw=blue, line width=4] (0,0) to[out=180,in=270] (-1,0.25) to[out=90,in=180] (0.5,0.5);
\path[preaction={clip, postaction={draw=red, line width=4,fill=red!30,opacity=0.5}}] (0,0) to[out=0,in=270] (1.5,0.25) -- (1.5,-0.75) to[out=270,in=0] (0.25,-1.5) to[out=180,in=270] (0,-1) -- cycle;
\path[draw=red, line width=4] (0,0) to[out=0,in=270] (1.5,0.25) to[out=90,in=0] (0.5,0.5);
\shade [ball color=black] (0,0) circle (0.2);
\shade [ball color=black] (0.5,0.5) circle (0.2);
\shade [ball color=gray] (0,-1) circle (0.2);
}
\tikz{
\clip (-1.5,-2) rectangle (2,2);
\path[preaction={clip, postaction={draw=blue, line width=4,fill=blue!30,opacity=1}}] (0.5,0.5) to[out=180,in=90] (-1,0.25) -- (-1,-0.5) to[out=270,in=180] (0.25,-1.25) to[out=45,in=270] (0.5,-0.25) -- cycle;
\path[preaction={clip, postaction={draw=red, line width=4,fill=red!30,opacity=1}}] (0.5,0.5) to[out=0,in=90] (1.5,0.25) -- (1.5,-0.5) to[out=270,in=0] (0.25,-1.25)  to[out=45,in=270] (0.5,-0.25) -- cycle;
\path[preaction={clip, postaction={draw=green, line width=4,fill=green!30,opacity=1}}] (0.5,0.5) to[out=0,in=90] (1.5,0.25) -- (1.5,1) to[out=90,in=0] (0.25,1.75) to[out=180,in=90] (-1,1) -- (-1,0.25) to[out=90,in=180] cycle;
\shade [ball color=black] (0.5,0.5) circle (0.2);
\path[preaction={clip, postaction={draw=blue, line width=4,fill=blue!30,opacity=0.75}}] (0,0) to[out=180,in=270] (-1,0.25) -- (-1,-0.5) to[out=270,in=180] (0.25,-1.25) to[out=180,in=270] (0,-0.75) -- cycle;
\path[preaction={clip, postaction={draw=red, line width=4,fill=red!30,opacity=0.75}}] (0,0) to[out=0,in=270] (1.5,0.25) -- (1.5,-0.5) to[out=270,in=0] (0.25,-1.25) to[out=180,in=270] (0,-0.75) -- cycle;
\path[preaction={clip, postaction={draw=green, line width=4,fill=green!30,opacity=0.5}}] (0,0) to[out=0,in=270] (1.5,0.25) -- (1.5,1) to[out=90,in=0] (0.25,1.75) to[out=180,in=90] (-1,1) -- (-1,0.25) to[out=270,in=180] cycle;
\shade [ball color=black] (0,0) circle (0.2);
}
\caption{ An iteration in the proof of Theorem~\ref{null-cobordant-converse}. Left to right: (1)~A~null-cobordant manifold $N = \partial X$ is represented by two black points. Two copies of~$X$, shown as blue and red curves, glue along~$N$ to form the double~$DX$, where $N$ is the entire 2-color class. (2)~The product $DX \times [0,1]$ from the proof of Lemma~\ref{doubleisboundary}. (3)~Folding the two halves of $DX \times \{0\}$ gives a manifold~$Z$ bounded by $\partial Z = DX$. (4)~Applying Lemma~\ref{double}, the 3-color class in the double~$DZ$ is the 2-color class in~$\partial Z$, which is~$N$ by induction.}
\label{doubling}
\end{figure} 

First, we check that the triangulation of~$Z$ is simplicial, with distinct simplices having distinct sets of vertices. Indeed, a face of $DX \times [0,1]$ with all vertices in the identified parts $(DX \times \{0\})$ and $(\partial X \times [0,1])$ must be fully contained in one of these parts, since no edge connects $\partial X \times \{1\}$ to $(DX \setminus\partial X) \times \{0\}$ by the above assumption.

Next, we verify that $Z$ is combinatorial, with vertex links that are balls or spheres. Let $v \in DX \times\{0\}$, so that $v \sim \tau(v)$ in $Z$. Its link consists of two glued parts, $\textrm{lk}_Z(v) = B \cup \tau(B)$ where $B = \textrm{lk}_{X \times [0,1]}(v)$ and $\tau(B)$ is the mirror image in other half of~$Z$. Each of the parts $B$ and $\tau(B)$ is a ball, as a link of a boundary vertex in a combinatorial manifold, and embeds in $Z$ since it comes from one half. Their boundaries are two spheres $S$ and $\tau(S)$, where $S = \partial B = \mathrm{lk}_{\partial(X \times [0,1])}(v)$. Note that $S \subset (X \times \{0\}) \cup (\partial X \times [0,1])$, since this set includes all boundary simplices that contain $v \in X \times \{0\}$, and so $S$ and~$\tau(S)$ are completely identified in~$Z$. On the other hand, the interiors of $B$ and $\tau(B)$ are disjoint from glued simplices since non-boundary faces that contains $v = (x_0,0)$ must have at least one vertex $(x_1,1)$ where $x_1 \not \in \partial X$. In conclusion, $\mathrm{lk}_Z(v) = B \cup \tau(B)$ is a combinatorial sphere.

The link $\mathrm{lk}_{DX \times [0,1]}(v)$ of a vertex $v \in DX \times \{1\}$ is a combinatorial ball, as $v$ is a boundary vertex. We show that this ball embeds in $Z$ under the gluing of $X \times \{0\}$ with $\tau(X) \times \{0\}$. If $v \not\in \partial X \times \{1\}$ then $v$ only has neighbors in its half, either $X \times [0,1]$ or $\tau(X) \times [0,1]$. Otherwise, by our choice of the product's triangulation, a vertex $v = (x_1,1)$ with $x_1 \in \partial X$ has neighbors $(x_0,0)$ only for $x_0 \in \partial X$, and these embed under $DX \times [0,1] \to Z$. In conclusion, $\mathrm{lk}_Z(v)$ is a combinatorial ball for each $v \in DX \times \{1\}$ and otherwise a combinatorial sphere, so $Z$~is a combinatorial manifold with boundary~$DX$ as required.
\end{proof}

\begin{proof}[Proof of Theorem~\ref{null-cobordant-converse}]
Null-cobordism is necessary by Lemma~\ref{null-cobordant}. We show it is sufficient by induction on~$k$. The case $k=1$ is trivially true. $N$~itself is the all-color class in a coloring of~$N$ with one color. For larger~$k$, we use the assumption that $N = \partial X$ for some manifold~$X$.

In the case $k=2$, we use the double~$DX$ of~$X$, where the two copies of $X$ have different colors, as in the leftmost part of Figure~\ref{doubling}. Lemma~\ref{double} with $k=1$ describes in detail such a triangulation and 2-coloring of~$DX$, where the 2-color class is homeomorphic to~$\partial X = N$.

For $k=3$, we first apply Lemma~\ref{doubleisboundary} to show that also $DX$ is a boundary $\partial Z$ for some manifold~$Z$, illustrated in the next two parts of Figure~\ref{doubling}. By the case $k=2$, we 2-color~$Z$ such that the 2-color class of $\partial Z$ is homeomorphic to~$N$. Lemma~\ref{double} with $k=2$ gives a triangulation and a 3-coloring of the double~$DZ$, where the 3-color class is homeomorphic to~$N$. See the last part of Figure~\ref{doubling}.

For general $k$, we iterate this construction, each time taking a $(k-1)$-colored manifold that gives~$N$, adding a null-cobordism using Lemma~\ref{doubleisboundary}, and $k$-coloring the double as in Lemma~\ref{double}. This gives a $k$-colored closed manifold where $N$ is the all-color class.
\end{proof}

Remark~\ref{sw-numbers} implies that Theorem~\ref{null-cobordant-converse} is equivalent to the following.
\begin{corollary}
A manifold~$N$ occurs as the all-color class in some closed $k$-colored manifold~$M$ with $k \geq 2$ if and only  if  all the Stiefel-Whitney numbers of $N$ are trivial.
\end{corollary}

Stiefel-Whitney classes play a more essential role in our next question. Question~(I) have asked which submanifolds~$N$ arise from our coloring construction, without placing any restriction on the ambient~$M$. We now consider what submanifolds can occur as submanifolds of a fixed ambient manifold~$M$. Natural choices for $M$ include a sphere~$S^d$, or  a ball~$B^d$ when considering submanifolds with boundary.

\begin{question*}[\textbf{II}]
Given a combinatorial manifold~$M$ and an integer $k \geq 2$, what manifolds~$N$ of codimension $(k-1)$ arise as the entire $k$-color class from some triangulation and coloring of~$M$?
\end{question*}

Clearly the answer depends on embeddability. If there is no embedding of $N$ in~$M$ then $N$~cannot arise even as a component of a $k$-color class. For example, a closed nonorientable $d$-manifold does not embed in~$S^{d+1}$ and cannot be a 2-color class. Similarly, a lens space $L(p,q)$ does not embed in~$S^4$, so it does not occur as a 2-color class. Note that Lemma~\ref{null-cobordant} does not give an obstruction to the realization of lens spaces in any particular 4-manifold, since all orientable 3-dimensional manifolds are null-cobordant, as follows from work of Wallace~\cite{wallace1960modifications} and Lickorish~\cite[Theorem~3]{lickorish1962representation}.

The next proposition shows that nonorientability poses an  obstruction to realizing~$N$ as a $k$-color class in an orientable manifold~$M$, even when $N$ does embed in~$M$. Thus, a~Klein bottle can be embedded in~$S^4$, but is not a 3-color class in any 3-coloring of $S^4$. Similarly, for manifolds with boundary, a M\"obius strip is not a 2-color class in a 2-colored solid torus, nor a 3-color class in a 3-colored~$B^4$. Moreover, a M\"obius strip is not a component of any of the 2-color classes in a 3-colored~$S^3$.

\begin{proposition}
\label{orientation}
Let $N$ be a component of a $k$-color class in an orientable manifold~$M$. Then $N$ is orientable.
\end{proposition}

\begin{proof}
We use the two following facts.  First, the boundary of an orientable manifold is orientable. Second, a   codimension-zero submanifold of an orientable manifold is orientable. The proof proceeds by taking a sequence of nested submanifolds between $N$ and~$M$, each of which is either a boundary or a codimension-zero submanifold of the next one.

Let $N$ be a component of a $k$-color class in~$M$, and suppose $N$ is nonorientable. Consider a $(k-1)$-color class~$N'$ obtained by omitting one of the $k$ colors. The boundary~$\partial N'$ is nonorientable too, because by Lemma~\ref{submanifold} it contains~$N$ as a codimension-zero submanifold. Since $\partial N'$ is nonorientable, so is the nonclosed manifold~$N'$. Again, the $(k-1)$-color class~$N'$ is part of the boundary of a $(k-2)$-color class~$N''$. This manifold~$N''$ is also nonorientable, since it has a codimension-one nonorientable manifold~$N'$ in its boundary. Repeating, we deduce that one of the 1-color classes, and hence also $M$ itself, is non-orientable, contradicting our assumption.
\end{proof}

Nonorientability is a special case of a Stiefel-Whitney class, which obstructs the realization of a manifold~$N$ as a submanifold of~$M$. The next proposition extends this obstruction to all Stiefel-Whitney classes. The $j$th Stiefel-Whitney class of the tangent bundle of a manifold~$N$ is a specific element of the $j$th cohomology, $w_j(N) \in H^j(N,\mathbb{Z}_2)$, see \cite{milnor1974characteristic}. A~manifold~$N$ is nonorientable if and only if $w_1(N) \neq 0$. Proposition~\ref{stiefelwhitney} shows that $w_j(N) \neq 0$ is an obstruction to realizability for any~$j$. Note that having nontrivial Stiefel--Whitney classes is not a special case of Remark~\ref{sw-numbers}, because there we consider the related but nonequivalent invariants, Stiefel--Whitney \emph{numbers}.

As a specific example, consider $N = \mathbb{C}P^2 \# \overline{\mathbb{C}P^2}$, the connected sum of a complex projective plane and a copy with an oppositely oriented copy. Since $N$ is null-cobordant, Theorem~\ref{null-cobordant-converse} does not pose an obstruction to realizing it as the all-color class in some manifold~$M$. For $d$ large enough, $N$~embeds in~$S^d$, by the Whitney Embedding Theorem, so  $N$ could potentially be realizable by a $(d-3)$-coloring of the $d$-sphere.  Proposition~\ref{orientation} does not give an obstruction, since $N$ is orientable. However, the next proposition shows that $N$ is not obtained from any coloring of $S^d$ since $w_2(S^d)=0$ while $w_2(N) \in H^2(N,\mathbb{Z}_2) = \mathbb{Z}_2 \oplus \mathbb{Z}_2$ is nontrivial.

\begin{proposition}
\label{stiefelwhitney}
Let $N$ be a component of a $k$-color class in a manifold~$M$ with trivial $j$th Stiefel-Whitney class, $w_j(M)=0$. Then $w_j(N)=0$.
\end{proposition}

\begin{proof}
As with orientation, each boundary component of a   manifold with trivial Stiefel--Whitney class~$w_j$ also has $w_j = 0$, and a codimension-zero submanifold of a manifold with $w_j =0$ also has $w_j = 0$. These facts follow from the naturality of Stiefel--Whitney classes and from the Whitney product formula~\cite{milnor1974characteristic}. 

The rest of the proof is similar to that of Proposition~\ref{orientation}. If a component of a $k$-color class satisfies $w_j(N) \neq 0$, then so does a $(k-1)$-color class~$N'$ having $N$ as one of its boundary components, and also a $(k-2)$-color class $N''$, and so on all the way up to~$M$. This yields a contradiction since it implies $w_j(M) \neq 0$. 
\end{proof}

\medskip
Finally, we consider the question of realizing a submanifold~$N$ when not just $M$ and~$N$ are specified, but also the isotopy class of the embedding of $N$ in~$M$. 

\begin{question*}[\textbf{III}] \label{III}
Given an embedding $N \subset M$ of codimension $k-1$, can it be obtained up to ambient isotopy as the entire $k$-color class for some $k$-coloring and triangulation of $M$? Can it be a component of the $k$-color class? 
\end{question*}

We first discuss the following example, which investigates how this question differs from Question~(II).

\begin{example}
The Klein bottle occurs as the 3-color class in a 4-manifold $M$ given as in Theorem~\ref{null-cobordant-converse}. The constructed $M$ is non-orientable and contains a codimension-one submanifold that is the union of two solid Klein bottles. Now consider a second embedding of the Klein bottle into $M$ in which it is contained in a small 4-ball. This embedding is not a 3-color class in any 3-coloring of $M$, since in such a coloring a neighborhood of the Klein bottle within the 4-ball would necessarily be non-orientable, by the argument of Proposition~\ref{orientation}. 
\end{example}

Reasoning similar to this example combines with Proposition~\ref{stiefelwhitney} to show that for any embedding $f:N \hookrightarrow M$ induced by a coloring, all Stiefel-Whitney classes $w_j(N)$ are in the image of $f^*:H^*(M;\mathbb{Z}_2) \to H^*(N;\mathbb{Z}_2)$.
However, it is not clear whether an embedding $f:N \to M$ has a realization as a component of a $k$-coloring whenever all the Stiefel-Whitney classes of $w_j(N)$ are in the image of $f^*:H^*(M;\mathbb{Z}_2) \to H^*(N;\mathbb{Z}_2)$.

\medskip
To approach Question~(III), it is useful to have a general method to realize smooth submanifolds as color classes of triangulations.

\subsection*{Coloring Stratified Manifolds} 
\label{stratified}

Our discussion has been framed around triangulated manifolds, but many of the results we obtain hold in a more general setting described by a class of stratified smooth manifolds. We now define these manifolds. In this setting, we assign colors to regions in a manifold, so that the local structure where regions intersect is identical to that determined by the Voronoi coloring of a triangulation.

\begin{definition}
\label{stratification}
A {\em $k$-colored stratification} of a smooth $d$-manifold $M$ is a decomposition of $M$ into a union of closed codimension-0 submanifolds $R_1, R_2,  \dots R_k$, each of which is assigned one of $k$ colors, with the following properties:
\begin{enumerate}
\item 
Each component of a color class of $r$ colors $C \subseteq R_{i_1} \cap R_{i_2} \cap \dots \cap R_{i_r} $ is a smooth codimension-$(r-1)$ submanifold of $M$, possibly with boundary. The boundary, if nonempty, is contained in the union of color classes of the original $r$ colors and one or more additional colors.
\item  
Each component $C$ as above has a normal tubular neighborhood $N(C) \subset M$ that is diffeomorphic to $C \times \Delta^{r-1}$, with the property that $C \times \{v_1\},\dots,C \times \{v_r\}$ for the $r$ vertices of~$\Delta^{r-1}$ lie in the regions $R_{i_1}, R_{i_2},  \dots, R_{i_r}$ with $r$ distinct colors.
\end{enumerate}
\end{definition}

Note that this structure is found in a neighborhood of $r$-color classes formed from Voronoi colorings of triangulated manifolds.

When a $k$-color stratified $d$-dimensional manifold $M$ also has a triangulation with an assignment of colors given to its vertices, we can compare the sets of submanifolds that appear as color classes in the stratification to the color classes corresponding to the triangulation.  If there is an isotopy of $M$ that carries the first set of color classes to the second, then we say that the two colorings are {\em compatible}. We now show how to start with a $k$-color stratified $d$-dimensional manifold~$M$ and construct a triangulation that gives a compatible coloring.

There is a standard way to triangulate a product of two simplicial complexes without introducing additional vertices. The following lemma describes this product and its $r$-color class for a product of a complex with an $r$-colored $(r-1)$-simplex. 

\begin{lemma}
\label{product}
Let $\Sigma$ be a simplicial complex with vertex set~$U$, and $\Delta$ an $(r-1)$-simplex with vertex set~$V$, and fix linear orderings on $U$ and~$V$. The product $\Sigma \times \Delta$ has vertex set $U \times V$. There is a triangulation of $\Sigma \times \Delta$ having vertices $U \times V$ and $d$-faces $\{(u_0,v_0), \dots, (u_d,v_d)\}$, with $d+1$ distinct ordered pairs satisfying $u_0 \leq u_1 \leq \dots \leq u_d$ and $v_0 \leq v_1 \leq \dots \leq v_d$, where each $u_i \in U$, $v_j \in V$, and $\{u_0,\dots,u_d\}$ is a simplex in $\Sigma$ when ignoring repetitions. 

Moreover, if the vertices of $\Delta$ are assigned $r$ colors and the vertex $(u_i,v_j)$ is assigned the color of~$v_j$, then the Voronoi construction gives rise to an $r$-color class isotopic to $\Sigma \times p$, where $p$ is the barycenter of~$\Delta$.
\end{lemma}

For example, the product of a triangle $ABC$ and an interval $[0,1]$ is a prism with three 3-simplices:
$\{A_0,A_1,B_1,C_1\}$, $\{A_0,B_0,B_1,C_1\}$, and $\{A_0,B_0,C_0,C_1\}$. If we color $[0,1]$ so that $0$ is red  and $1$ is blue, then the vertices 
$\{A_0,B_0,C_0 \}$ are red and the vertices 
$\{A_1,B_1,C_1\}$ blue. The red-blue color class in the product is isotopic to the triangle $ABC \times (1/2)$ separating the two triangular ends of the prism. See the proof of Lemma~\ref{double} and Figure~\ref{prisms} there for a more detailed demonstration.

We apply Lemma~\ref{product} to triangulate the product of a triangulated manifold  $C$  with a simplex $\Delta$. Note that the  manifold $C$ may have boundary, in which case the triangulation on  $\partial C \times \Delta$ coincides with the triangulation obtained by applying  Lemma~\ref{product} to
$\partial C$ directly. 

\begin{proposition}
\label{strata}
Let $M$ be a manifold with a $k$-colored stratification. Then $M$ has a compatible $k$-colored triangulation.
\end{proposition}

\begin{proof}
Starting with a $k$-colored stratification of~$M$, we construct a compatible $k$-colored triangulation. 

Assume first that the intersection of all $k$ colors $M_k \ne \varnothing$. Then a component $C$ of~$M_k$ is a closed submanifold of~$M$ of codimension~$k-1$. $C$ has an embedded tubular neighborhood of radius $\epsilon >0$. Inside a normal neighborhood of $C$ of radius $\epsilon/3$ there exists a smaller neighborhood that is homeomorphic to $C \times \Delta^{k-1}$, with  the vertices of $\Delta^{k-1}$ lying in distinct colored regions. Fix a triangulation of $C$, and assign to it an arbitrary order. Then extend the triangulation to a triangulation of $C \times \Delta^{k-1}$ as in Lemma~\ref{product}.  Since the tubular neighborhood radius is less than $\epsilon/3$, we can do this simultaneously for all components of~$M_k$ without creating any overlaps.
Assigning to each vertex the color of the region in which it lies, we obtain a $k$-color class isotopic to $C$.  We do the same for other components of $M_k$.

Now consider a component $C_1$ of $M_{k-1}$, the intersections of $k-1$ color regions. If   $\partial C_1 =   \varnothing$ then we proceed as before. Otherwise $\partial C_1 \subset M_k$ and we can extend the triangulation of~$\partial C_1$ to a triangulation of $C_1$ and extend the previous vertex order to an arbitrary order on the additional vertices. Along the boundary,
$\partial C_1  \times  \Delta^{k-2} $ inherits a triangulation from the previous step. Using Lemma~\ref{product} we triangulate $  C_1 \times  \Delta^{k-2} $.  This triangulation agrees with the previous triangulation on  $\partial C_1  \times  \Delta^{k-2} $.  We continue in this way for color classes with decreasing numbers of colors.

Continuing in this way we construct a triangulation of all of $M$. Each vertex inherits a color from the region in which it lies, and each stratum meeting $r$ colors coincides with the $r$-color class induced from the triangulation.
\end{proof}

We now return to Question (III) which asks which embeddings of one manifold in another can be obtained as a $k$-color class.
We first answer it for the case of 1-manifolds in~$S^3$.

\begin{proposition}
\label{link}
Every link in the 3-sphere is obtained as the $3$-color class of some $3$-colored triangulated~$S^3$.
\end{proposition}

\begin{proof}
Take three parallel copies of a Seifert surface spanning the link. These three surfaces separate the link-complement into 3 regions, each of which is assigned a different color. This defines a $3$-colored stratification, and Proposition~\ref{strata} constructs a compatible triangulation. The original link is isotopic to the $3$-color class of this triangulation.
\end{proof}

Proposition~\ref{tiling} in Section~\ref{genus4} gives an efficient alternate way to construct such 3-colored 3-spheres, based on a planar diagram of the link. That construction realizes the link in a 3-colored~$S^3$ whose vertices lie on the boundary of a 4-dimensional box in~$\mathbb{Z}^4$. 

It follows from Proposition~\ref{link} that a 3-coloring exists in which the intersection of any 2 colors realizes the knot genus. An extension of this procedure 
can be used to generate a surface in $B^4$ realizing the 4-ball genus of a knot or link.

\begin{proposition} \label{3color4ball}
The previous construction of a  3-coloring of a triangulation of $S^3$ giving a chosen link $L$ can be extended from $S^3 = \partial B^4$ to $B^4$ so that the intersection of all three colors in $B^4$ realizes the 4-ball genus of the link.
\end{proposition}

\begin{proof}
We give two proofs, one combinatorial and one based on area minimization.

Proof 1: Start with a link $L$ in $S^3 = \partial B^4$. As before take any three disjoint Seifert surfaces, $S_1, S_2, S_3$ with $\partial S_i = L$.  We can choose all three to be parallel in $S^3$.  Let $F \subset B^4$ be a minimal genus oriented surface in the 4-ball with $\partial F = L$, so that $F$ realizes the 4-ball genus of $L$. We form three closed surfaces in $B^4$:  $T_i = S_i \cup F, i = 1,2,3.$ The intersection of any two of these three surfaces is their common subsurface $F$.  Since $H_2(B^4, \partial B^4) = 1$, we can find three 3-dimensional manifolds with
boundary $M_i$, properly embedded in $B^4$, with  $\partial M_i = T_i$ ~\cite{thom1954quelques}. 
Each of the $M_i$ has zero framing along $F$, as shown in \cite{kervaire1965higher, kirk1991generalized}, so we can isotope them in a neighborhood of $F$ so that their intersection with this neighborhood coincides with $F$.
It remains to  show that the $M_i$  can be chosen to have disjoint interiors. 

Suppose that $M_1 \cap M_2 \ne F$. Applying transversality, the intersection can be assumed after a slight perturbation to be the disjoint union of $F$ and a finite collection of closed surfaces.
The union of $M_1$ and $M_2$, with appropriate orientation, represents a null-homologous cycle in $H_3(B^4, \partial B^4)$. If $M_i \cap M_j \ne \emptyset$ then there is a 4-dimensional
submanifold $X \subset B^4$  with $\partial X \subset M_i \cup M_j$ and int$(X)$ disjoint from $F$.  Replacing  $\partial X \cap M_1 $ with  $\partial X\cap M_2$ and perturbing slightly produces new 3-dimensional submanifolds, homologous to $M_1$ and~$M_2$, that intersect along $F$ and possibly other surfaces. The intersection of these new 3-manifolds contains fewer  surface components than did $M_1 \cap M_2$. Moreover this process can be done leaving $M_2$ unchanged, only replacing submanifolds of $M_1$ with submanifolds of $M_2$, and only perturbing $M_1$. Continuing in this way we obtain new 3-manifolds $M_1'$ and $M_2$ intersecting only along $F$. 
We then repeat the process with $M_2$ and $M_3$, obtaining  3-manifolds $M_3'$ and $M_2$ intersecting only along $F$. 
Now the intersection of $M_2$ with each of $M_1'$ and $M_3'$ is precisely $F$.  A final application of the process to $M_1'$ and $M_3'$ results in all three 3-manifolds being disjoint away from their common intersection along $F$.

We have found three disjoint 3-manifolds intersecting along $F$ that divide the 4-ball into three regions that satisfy the conditions of Proposition~\ref{strata}. We  apply this proposition to triangulate the 4-ball and color the vertices so as to obtain the surface $F$ as the sole component of the 3-color class.  Thus the 3-color class realizes the 4-ball genus.

Proof 2: In this second approach we use area minimization properties rather than cut and paste arguments to make the $M_i$ disjoint. Take each $M_i$ to be an area minimizing 3-manifold in $B^4$ among all homologous 3-manifolds sharing the same boundary.  Such minimizers exist, with each $M_i$ a properly embedded submanifold and any pair  $M_i, M_j $  are either disjoint except where their boundaries intersect along $F$ or exactly coincident \cite{morgan2016geometric}.  Since  the surfaces $S_i \subset \partial M_i$ are distinct, the 3-manifolds must also be distinct.
We now apply Proposition~\ref{strata} as before.
\end{proof}

This argument applies in more general settings.
For example, a knotted sphere  $K^{d-2} \subset S^{d}$ has a trivial 2-dimensional normal bundle  \cite{kervaire1965higher}. It can be realized by a 3-coloring of $S^{d}$ that results in three ${d-1}$-dimensional Seifert surfaces in $S^d$ that span the knot. An identical argument to the first proof of Proposition~\ref{3color4ball} yields the following.

\begin{proposition}
\label{seifertd}
Suppose  that $K^{d-2}$ is a  $( d-2)$-dimensional knot or link in $S^d = \partial B^{d+1}$ that bounds a manifold $F^{d-1} \subset B^{d+1}$.  Then there is a subdivision and coloring of $B^{d+1}$ whose
3-color class is isotopic to $F$.
\end{proposition}

\section{Random Submanifolds}
\label{random}

In this section we examine the resulting color classes occurring when  coloring the vertices of a triangulation at random. The discrete nature of this construction of random submanifolds makes some of their topological and geometric properties accessible both for theoretical investigation and practical computations.

\medskip

To define the random model, consider a $d$-dimensional triangulated manifold~$M$. As shown in Lemma~\ref{manifold_in_M}, the all-color class in a $k$-coloring of~$M$ is a proper submanifold of codimension~$k-1$. Therefore, for any desired dimension $m < d$, we can color the vertices of $M$ with $k=d-m+1$ colors, $\mathcal{C} = \{C_0,C_1,\dots,C_{d-m}\}$,
and obtain
a random $m$-dimensional proper submanifold of~$M$. This yields a probability distribution over various topological types of $m$-dimensional manifolds, and moreover, over the isotopy classes of their embeddings in~$M$. Many different submanifolds arise in this model, as discussed in Section~\ref{universality}.

The simplest strategy for generating a random coloring is to assign to every vertex a uniformly random color from~$\mathcal{C}$, independently of the other vertices. We may also use a nonuniform probability vector $(p_0,p_1,\dots,p_{d-m})$ as a distribution on~$\mathcal{C}$. More involved models may make the probability vector location-dependent, or introduce dependencies between nearby vertices. Another possibility is a stochastic process where the coloring is dynamically changing. 

If the ambient manifold has a boundary~$\partial M$, then we can impose a coloring on the boundary vertices that sets special conditions there. Such boundary conditions may be used to control the boundary of the all-colors submanifold, or to guarantee a nontrivial or large component.

\medskip

The following definition lists a selection of standard choices for the ambient triangulated manifold~$M$, such that its vertices conveniently lie on an integer grid~$\mathbb{Z}^d$. These constructions have been considered in previous works~\cite{sheffield2014tricolor, de2019random, bobrowski2020homological}, sometimes up to a linear transformation in~$\mathbb{R}^d$, and sometimes equivalently in terms of the dual \emph{permutahedral tessellation} and the coloring of its $d$-dimensional tiles.

\begin{samepage}
\begin{definition}
\label{standard}
Let $d,n \in \mathbb{N}$. 
\begin{itemize}
\itemsep0.5em
\item
The \emph{$d$-ball}, or \emph{$d$-disk}, or \emph{$d$-cube} of order $n$, denoted $B^d_{n}$, is the space $[0,n]^d$, with $n^dd!$ top simplices, where each of the $n^d$ unit cubes is subdivided as in Figure~\ref{braid}.
\item
The \emph{$d$-sphere} of order $n$, denoted $S^d_n$, is the triangulated space obtained from the $(d+1)$-ball by taking the boundary~$\partial B^{d+1}_n$.
\item
The \emph{$d$-space}, denoted $R^d$, is the infinite triangulation of $\mathbb{R}^d$, with vertices $\mathbb{Z}^d$ and each unit cube is triangulated with $d!$ top simplices as in the $d$-ball above.
\item
The \emph{$d$-torus} of order $n$, denoted $T^d_n$, is the quotient of the $d$-space $R^d/n\mathbb{Z}^d$, or equivalently $B_n^d$ with opposite $(d-1)$-faces identified.
\end{itemize}
\end{definition} 
\end{samepage}

The parameter $n$ in the above triangulations of the ball, sphere and torus, allows one to obtain increasingly complex submanifolds of the ambient $d$-dimensional space. Letting $n \to \infty$ we can explore the limiting properties and large-scale behavior of typical random submanifolds.

Given any triangulated manifold~$M$, we would like to consider a similar sequence of refined triangulations, that would yield all achievable submanifolds of~$M$ and capture their asymptotic nature. One solution is to refine a $d$-dimensional manifold~$M$ by iterating the barycentric subdivision $n$ times. This scheme yields arbitrarily small simplices, but the vertex degree explodes, as  each iteration multiplies it by up to~$d!$. The cubing viewpoint suggests an alternative subdivision scheme, as follows. 

\begin{samepage}
\begin{definition}
\label{scheme}
Let $M$ be a $d$-dimensional triangulated manifold. Denote by~$M_n$ the following sequence of refinements of $M$.
\begin{itemize}
\itemsep0.5em
\item Subdivide every $d$-simplex of $M$ into a $d+1$ cubes as described in Section~\ref{construction}.
\item Subdivide every $d$-cube into $n^d$ subcubes of side length $1/n$.
\item Subdivide every small cube into $d!$ simplices using the braid triangulation. 
\end{itemize}
\end{definition}
\end{samepage}
Note that cubical faces shared by several cubes within a simplex are triangulated consistently by this procedure, and also that this triangulation is consistent on restrictions to lower-dimensional faces of a simplex, and therefore may be applied to a whole simplicial complex. In contrast with the iterated barycentric subdivision, Definition~\ref{scheme} gives vertices and faces of bounded valence. This scheme replaces every $d$-dimensional simplex with $n^d(d+1)!$ simplices, in a way that preserves, in some sense, the geometry of the space, without introducing much ``distortion''.

\medskip

Each of the following sections, \ref{knots}, \ref{asymptotic}, and~\ref{genus4} focuses on another important application of our model of random submanifolds, carried out in different settings and contexts.

\section{Random Knots}
\label{knots}

We  now examine  submanifolds obtained by randomly 3-coloring the interior vertices of the  3-cube $B_n^3$, which is the model of random knots introduced by de Crouy-Chanel and Simon~\cite{de2019random}. The coloring of the boundary vertices is fixed using the following rule, which is sometimes called the \emph{beach ball} coloring:
\begin{samepage}
\begin{multicols}{2}
\begin{itemize}
\itemsep0.5em
\item Blue: $x=0$ or $z=n$
\item Green: $z=0$ or $y=n$
\item Red: $y=0$ or $x=n$ 
\end{itemize}
\columnbreak
\begin{center}
\begin{tikzpicture}[scale=1.25,line width=1.5, line cap=round]
\fill[green!50] (-0.5,0.3) -- (-0.5,1.3) -- (0.5,1.3) -- (0.5,0.3) -- cycle;
\fill[green!50] (-0.5,0.3) -- (0.5,0.3) -- (1,0) -- (0,0) -- cycle;
\fill[red!50] (0.5,1.3) -- (0.5,0.3) -- (1,0) -- (1,1) -- cycle;
\draw[green] (0,0) -- (1,0) (1,0) -- (0.5,0.3) -- (-0.5,0.3) (0.5,0.3) -- (0.5,1.3);
\fill[blue!25,opacity=0.75] (0,0) -- (0,1) -- (-0.5,1.3) -- (-0.5,0.3) -- cycle;
\fill[blue!25,opacity=0.75] (0,1) -- (1,1) -- (0.5,1.3) -- (-0.5,1.3) -- cycle;
\fill[red!50,opacity=0.75] (0,0) -- (0,1) -- (1,1) -- (1,0) -- cycle;
\draw[red] (1,0) -- (1,1) -- (0.5,1.3);
\draw[blue] (0,0) -- (-0.5,0.3) -- (-0.5,1.3) -- (0.5,1.3);
\draw[blue] (0,1) -- (-0.5,1.3);
\draw[red] (0,0) -- (0,1) (0,1) -- (1,1); 
\node at (0,0)[circle,fill=green,inner sep=1.2]{};
\node at (0.5,1.3)[circle,fill=green,inner sep=1.2]{};
\end{tikzpicture}
\end{center}
\end{multicols}
\end{samepage}
On the edges where two colors meet, such as red and blue for $(0,0,z)$, one may choose one of the two colors arbitrarily. For concreteness, we use the following rule: blue beats green beats red beats blue. This leaves only the vertices $(0,0,0)$ and $(n,n,n)$ which we color green. These two vertices are contained in two respective boundary triangles whose center points give the 3-color class of the boundary. The random knot is obtained by randomly coloring the interior vertices, taking the 3-color class component that reaches the boundary, and forming a loop by connecting its ends along the boundary of $B_n^3$. Figure~\ref{figure8} shows a coloring of the grid $[0,\dots,4]^3$ with beach ball boundary conditions in which the 3-color curve is a figure-eight knot.

Every knot type arises from some coloring of $B_n^3$ for $n$ large enough. This important property of the random model is not obvious from its definition in~\cite{de2019random}, but follows from our argument in Proposition~\ref{link} above, and more explicitly from the construction in Proposition~\ref{tiling} below. Moreover, these methods show that a given knot can even be realized as the unique component of the 3-color class. When realized as the only component, each of the 2-color classes determines the zero-framing of the knot.

\setcounter{figure}{7}
\begin{figure}
\centering
\vspace{0.5em}
\includegraphics[width=0.6\textwidth]{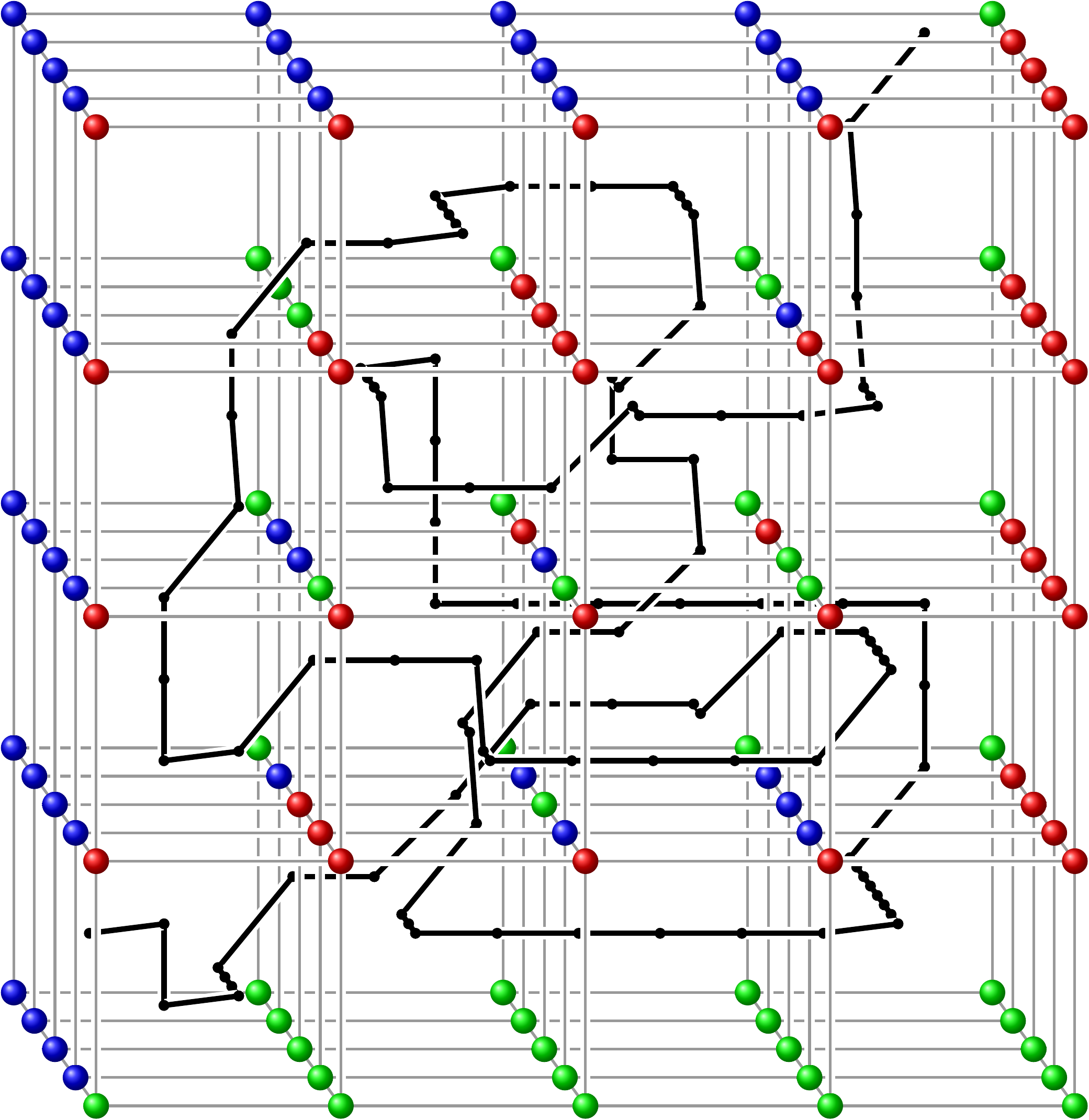}
\vspace{0.5em}
\caption{A coloring of the $5 \times 5 \times 5$ grid that yields a figure-eight knot. The center points (black dots) of 3-colored triangles are connected (black segments) inside common tetrahedra. For clarity, the 1-skeleton of the triangulation (gray lines) is shown with diagonal edges omitted. See the six 3-simplices in each small cube in Figure~\ref{braid}. Here, the vertex $(0,0,0)$ is the green one at the bottom left. This knotted 3-color curve is one of six nontrivial results out of $10^{10}$ random colorings.
}
\label{figure8}
\end{figure}

De Crouy-Chanel and Simon have conducted an experimental investigation of the resulting knots, and have posed several conjectures. Their experiments indicate that for the random knots that appear in a large three-dimensional cube, the frequency of the unknot goes to zero as the size of the cube increases \cite[Table~1]{de2019random}. They note that it ``would be very interesting to prove these behaviors rigorously and have exact values for the scaling exponents.'' 

This basic property that most generated knots are knotted, experimentally observed by De Crouy-Chanel and Simon in this knot model, is known in general as the Debruck--Frisch--Wasserman conjecture. It was considered already around 1960, for knots that are generated by polygonal random walk in $\mathbb{R}^3$~\cite{delbruck1961knotting, frisch1961chemical}, and has been proven true in those models~\cite{sumners1988knots, pippenger1989knots, soteros1992entanglement, diao1994random, diao1995knotting}. Later, the analogous conjecture was verified by various methods in other models of random knots, such as random planar diagrams \cite{cantarella2016knot,chapman2016asymptotic2}, random petal diagrams \cite{even2016invariants,even2018distribution}, and random grid diagrams~\cite{witte2019link}. 

Here we verify the Debruck--Frisch--Wasserman conjecture for this model of random knots, resolving with a positive answer the question raised by de Crouy-Chanel and Simon. Our proof shows that a figure-eight connected summand is present in the relevant component of the 3-color class with high probability. Moreover, we show that the decay of the unknot probability is exponential in~$n$.

\begin{theorem}
\label{dfw}
The probability of the unknot in the de Crouy-Chanel--Simon model of random knots is~$\exp(-\Theta(n))$. 
\end{theorem}

\begin{proof}
Let $n \in \mathbb{N}$, and denote by $A$ the event that the random knot generated by this model in the 3-dimensional cube $B_n$ is trivial. 

To derive a lower bound on $P[A]$, we consider a specific unknotted curve. The interface between red and green boundary vertices goes along three edges of the large cube $B_n^3$, from $(0,0,0)$ to $(n,0,0)$ to $(n,n,0)$ to $(n,n,n)$. We consider the following set of interior vertices in the neighborhood of these edges:
$$ \{(1,1,1),(2,1,1),\dots,(n-1,1,1),\dots,(n-1,n-1,1),\dots,(n-1,n-1,n-1) \}$$
Let $A'$ denote the event that all the $3n-5$ vertices in this set turn out blue. This event determines the vertex-coloring in a path of unit cubes along the boundary from $(0,0,0)$ to $(n,n,n)$, such that each cube has two 3-colored triangles, shared with the previous one and the next one. The 3-color curve within these cubes is unknotted and parallel to the boundary. Hence, given $A'$, the resulting knot is necessarily the unknot. In other words, the event $A'$ is included in~$A$. 
In conclusion, with $r = 1/3^3$,
$$ P[A] \;\geq\; P[A'] \;=\; 3^{-3n+5} \;=\; \Omega(r^n) $$

On the other hand, we derive an upper bound on $P[A]$ that decays exponentially too. Consider $B_m = [0,m]^3$ as a subcube of~$B_n$ with the induced triangulation. The random coloring $c:\{0,1,\dots,n\}^3 \to \{red,green,blue\}$ restricts to a coloring of the vertices $\{0,1,\dots,m\}^3$. For every $m < n$, denote by $A_m$ the event that there exists a coloring of~$B_n$ that agrees with the restriction of~$c$ to~$B_m$ and gives an unknot. Note that $A_m$ only depend on the coloring of~$B_m$, and $A \subseteq A_m$ by taking the original coloring before the restriction, hence $P[A] \leq P[A_m]$.

Fix a coloring of $B_m$ for $m < n-5$. By the boundary coloring, the 3-colored curve that enters $B_m$ near $(0,0,0)$ must exit through one of the opposite faces of~$B_m$ where either $x$,~$y$, or $z$ are equal to $m$. Without loss of generality, for some $0 \leq i,j < m$ it passes through the middle of a 3-colored triangle of the following form:
$$ (m,i,j),\, (m,i,j+1),\, (m,i+1,j+1) $$ 
Indeed, the opposite sides of $B_m$ are triangulated by triangles of this form up to interchanging the roles of the three axes. Consider the event that the triangle shifted by $(1,1,1)$ also admits respectively the same three different colors:
$$ (m+1,i+1,j+1),\, (m+1,i+1,j+2),\, (m+1,i+2,j+2) $$ 
These two parallel triangles may be viewed as the bottom and the top of a small ``skew prism'', i.e., a triangle cross an interval. It is made of 3 tetrahedra from the triangulation of~$B_n$, and has three quadrilateral side faces, each of which only contains vertices of two different colors. Hence the 3-color curve entering from the bottom triangle must exit from the top one. We use this skew prism as a small connecting part, to take us one step away from the already colored boundary of $B_m$, and of $B_n$ if $i$ or $j=0$.

This shifted triangle may be completed to a whole copy of the $5 \times 5 \times 5$ configuration in Figure~\ref{figure8}. Indeed, the 3-color triangle between (0,0,0), (0,0,1), and (0,1,1) is matched to that triangle, interchanging colors or axes as needed. The opposite vertex (4,4,4) will end up at $(m+5,i+5,j+5)$ so that the whole configuration is included in the cube~$B_{m+5}$. In conclusion, at least one coloring of the 125 vertices 
$$ \{m+1,\dots,m+5\} \times \{i+1,\dots,i+5\} \times \{j+1,\dots,j+5\} $$ 
guarantees a figure-eight component in the random curve. This scenario rules out the possibility of an unknot, regardless of how the curve continues. It might happen independently for any fixed coloring of~$B_m$, and thus regardless of conditioning on~$A_m$. Therefore, the conditional probability $P[A_{m+5} | A_{m}]$ is bounded by $1-1/{3^{125}}$. Since by definition $A_{m+5} \subseteq A_m$, this implies for every $m < n-5$,
$$ P[A_{m+5}] \;\leq\; \left(1-\frac{1}{3^{125}}\right) P[A_{m}]$$
Starting with $P[A_0]=1$ and iterating $\left\lfloor(n-1)/5\right\rfloor$ times,
$$ P[A] \;\leq\; P\left[A_{5\left\lfloor\frac{n-1}{5}\right\rfloor}\right] \;\leq\; \left(1-\frac{1}{3^{125}}\right)^{\left\lfloor\frac{n-1}{5}\right\rfloor} P\left[A_0\right] \;=\; O(R^n) $$
where $R = \sqrt[5]{1-1/3^{125}}<1$ as required.
\end{proof}

\begin{remark}
It is plausible that for some positive $\alpha$ between $r$ and $R$ of the above proof $P[A] = \exp(-(\alpha \pm o(1))n)$. It may be interesting to show that such an exponent exists and to estimate it with a higher precision.
\end{remark}

By observing the empirical distribution of a sample of determinants, de Crouy-Chanel and Simon predicted that ``as $n$ increases the random knots are composite and contain more and more prime knots'' \cite[end of \S2.3]{de2019random}.
The following proposition, extending the arguments of Theorem~\ref{dfw}, gives a mathematical proof of this prediction, as well as additional asymptotic properties of the obtained knots.

\begin{samepage}
\begin{proposition}
\label{composite}
A random knot in the de Crouy-Chanel--Simon model satisfies the following properties with probability tending to one at an exponential rate as~$n \to \infty$.
\begin{enumerate}
\item 
The knot is composite.
\item 
The number of prime connected summands is at least linear in~$n$.
\item 
The knot determinant is at least exponential in~$n$.
\item 
For each fixed knot $K$, the number of copies of~$K$ that appear as connected summands is at least linear in~$n$.
\end{enumerate}
\end{proposition}
\end{samepage}

\begin{proof}
Let $E_m$ be the event that wherever the knot first exits the cube~$B_m$ there appears in $B_{m+5} \setminus B_m$ the $5 \times 5 \times 5$ color configuration that guarantees a figure-eight summand, as described in detail in the proof of Theorem~\ref{dfw} and Figure~\ref{figure8}. The events $E_0$, $E_5$, $E_{10}, \dots$ occur independently with probability $p=P[E_m]=1/3^{125}$. Hence the number of figure-eight connected summands is at least~$X_n$, for a binomial random variable $X_n \sim B(n/5,p)$ counting the number of occurring~$E_m$. 

The expectation is $\mathbb{E}[X_n] = np/5$, and by Chernoff's bound $P[X_n \leq np/10] < e^{-np/40}$. Therefore, except for an exponentially small probability, there are at least $np/10$ such figure-eight factors, and the knot determinant is at least~$5^{np/10}$, proving (1), (2) and~(3). Since any knot $K$ can appear in this model, an occurrence of $K$ as a connected summand is similarly guaranteed by some $k \times k \times k$ configuration, and (4) follows for every $K$.

However, the probability of these properties cannot converge to one faster than exponentially, since the unknot occurs with probability as least $3^{-3n}$ as shown above.   
\end{proof}

\medskip

We conclude this section with a discussion of our computer experiments in the regime of small~$n$. For $n=4$, we discovered the construction of Figure~\ref{figure8} by randomly coloring the $3^3=27$ interior vertices and computing the determinant for a quick nontriviality test. Over ${\sim}10^{10}$ trials in this small case, we obtained 6 results with $\det \neq 1$, none of which were equal or symmetric to any other. All of these colorings had $\det=5$ and were identified as the figure-eight knot. This was verified by three knot detection software packages: PyKnotId~0.5.3 \cite{pyknotid}, plCurve~8.0.6 \cite{plcurve}, and SnapPy~2.7 \cite{SnapPy}. The search required roughly $10000$ CPU hours.

\begin{table}[t]
\centering
\addtolength\tabcolsep{-1pt}
\begin{tabular}{c||c||c||c||cc||ccc||ccccccc||c}
$n$ & sample & $3_1$ & $4_1$ & $5_1$ & $5_2$ & $6_1$ & $6_2$ & $6_3$ & $7_1$ & $7_2$ & $7_3$ & $7_4$ & $7_5$ & $7_6$ & $7_7$ & $(\ast)$ \\ \hline \hline
4 & \rule{0pt}{2.6ex}$10^{10}$ & 0 & 6 & 0 & 0 & 0 & 0 & 0 & 0 & 0 & 0 & 0 & 0 & 0 & 0 & 0 \\
5 & $10^{9}$ & 1959 & 541 & 0 & 2 & 1 & 0 & 0 & 0 & 0 & 0 & 0 & 0 & 0 & 0 & 0 \\
6 & $10^{8}$ & 5228 & 1298 & 0 & 16 & 8 & 0 & 0 & 0 & 0 & 0 & 0 & 0 & 0 & 0 & 0 \\
7 & $10^{7}$ & 3301 & 927 & 0 & 34 & 5 & 0 & 0 & 0 & 0 & 0 & 0 & 0 & 0 & 0 & 1 \\
8 & $10^{7}$ & 11462 & 3725 & 3 & 193 & 79 & 4 & 3 & 0 & 2 & 0 & 2 & 0 & 0 & 0 & 5 \\
9 & $10^{7}$ & 28846 & 10307 & 15 & 993 & 416 & 29 & 26 & 0 & 15 & 0 & 5 & 0 & 9 & 11 & 53 \\
10 & $10^{7}$ & 59552 & 24160 & 116 & 2855 & 1317 & 179 & 110 & 0 & 64 & 1 & 28 & 6 & 52 & 39 & 114
\end{tabular}
\vspace{1em}
\caption{Occurrences of nontrivial knots in randomly colored $[0,n]^3$ grids. Curves giving determinant 3 are counted in the column $3_1$, and those with determinant 5 and above are sorted using their Alexander polynomial. The last column $(\ast)$ accounts for curves such as $8_1$ or the connected sum $3_1\#3_1$, whose invariants do not fit any prime knot with up to 7 crossings.}
\label{table}
\end{table}

We similarly generated and identified random knots in $[0,n]^3$ for $n \in \{5,6,\dots,10\}$. This experiment showed that despite the asymptotic result of Theorem~\ref{dfw}, the unknot is still highly dominant, occurring over $99\%$ of the time in these grid sizes, followed by the trefoil knot $3_1$, and then the figure-eight knot $4_1$. There also appeared to be a significant tendency toward the series of twist knots, $4_1, 5_2, 6_1, 7_2, 8_1, \dots$, which seemed considerably more prevalent than their counterparts of the same crossing numbers. See Table~\ref{table}.

The last observation may be relevant to biologists that study and compare different models of knotting for DNA filaments that are extremely confined in small volume, see for example~\cite{arsuaga2005dna,micheletti2008simulations}. The crossing number is often taken by default as a complexity measure for knot types. But more specific knot properties may play an important role depending on particular circumstances such as spatial confinement.

\section{Euler Characteristic and Percolation}
\label{asymptotic}

We now study the properties of random submanifolds in the general setting of Definition~\ref{scheme}. For $k \leq d+1$, a~closed triangulated $d$-manifold~$M$ is randomly $k$-colored with probabilities $(p_1,\dots,p_k)$ independently for each vertex.
As $n \to \infty$,
the subdivision $M_n$ locally resembles the grid~$\mathbb{Z}^d$ except for a vanishingly small fraction of the vertices. What asymptotic topological properties are possessed by the random submanifolds $N_n \subset M_n$ that arise as all-color classes and as other $m$-color classes? 

\subsection{The Expected Euler Characteristic}
\label{eec-section}

We first investigate the distribution of the Euler Characteristic, $\chi(N_n) \in \mathbb{Z}$, which is the most basic and classical topological invariant of a manifold~$N$. The distribution of the Euler Characteristic in random manifolds has been studied in various settings. One example are submanifolds arising from zero sets of random fields, see e.g.~\cite{adler1981geometry, podkorytov1998euler, burgisser2007average, adler2007random,  letendre2016expected}. Our discussion of the expected Euler Characteristic extends one of the combinatorial settings considered by Bobrowski and Skraba~\cite{bobrowski2020homological}. 

\begin{samepage}
\begin{definition}
\label{eec}
Let~$M_n$ be a sequence of triangulated manifolds, and let $\{C_1,\dots,C_k\}$ be a set of $k$ colors. The \emph{expected Euler Characteristic density} of a nonempty subset $\mathcal{C} \subseteq \{C_1,\dots,C_k\}$ is
$$ \mathcal{E}_{\mathcal{C}}(p_1,\dots,p_k) \;=\; \lim_{n \to \infty}\; \frac{\mathbb{E}\left[\chi(N_n)\right]}{ |\text{vertices}(M_n)|} $$
where $N_n \subseteq M_n$ is the color class corresponding to~$\mathcal{C}$ in a random $k$-coloring of the vertices of~$M_n$, independently according to the probability vector $(p_1,\dots,p_k)$.
\end{definition}
\end{samepage}

In \cite{bobrowski2020homological}, Bobrowski and Skraba studied $\mathcal{E}_{C}(p,1-p)$ for a single color~$C$ in a $2$-coloring of the flat torus~$T^d$, across various discrete and continuous models. In our random model the $m$-color classes are well-behaved manifolds for all~$m \leq k$. We hence extend the study of the expected Euler Characteristic density to all $m$-color classes in $k$-colored manifolds, for general $m$ and $k$. 

In the following theorem, $M$~is any fixed triangulated $d$-dimensional manifold. The triangulated manifolds $M_n$ are a sequence of refinements of~$M$ obtained by cubings with length parameter~$1/n$, as in Definitions~\ref{standard}-\ref{scheme}. Their random submanifolds $N_n \subset M_n$ arise as color classes, possibly having non-empty boundary. For the sake of concreteness, one may have in mind the triangulated flat tori~$M_n = T_n^d$. However, the formula does not depend on the topology of~$M$.

\begin{theorem}
\label{ec}
Let $M$ be a $d$-dimensional triangulated closed manifold. Given a choice of colors $\mathcal{C} \subseteq \{C_1,\dots,C_k\}$, the expected Euler Characteristic density of the corresponding $(d-|\mathcal{C}|+1)$-dimensional submanifold $N_n$  in a randomly $k$-colored $M_n$ is
$$ \mathcal{E}_{\mathcal{C}}(p_1,\dots,p_k) \;=\; \sum_{r=|\mathcal{C}|}^{d+1} \genfrac{\{}{\}}{0pt}{}{d+1}{r} \, (r-1)! \; \sum_{X \subseteq\, \mathcal{C}}(-1)^{|X|}\left(-\sum_{C_i \in X}p_i\right)^r \;.$$
Here $\displaystyle \genfrac{\{}{\}}{0pt}{}{d+1}{r} = \sum\limits_{s=1}^r \frac{(-1)^{r-s} s^d}{(s-1)!(r-s)!}$ are the Stirling numbers of the second kind. 
\end{theorem}

Before giving some examples and a proof, we remark that there are some deterministic dependencies between these Euler Characteristics. There are $2^k-1$ Euler Characteristics, one for each nonempty subset of the $k$ colors. The Euler Characteristic of every odd-dimensional color class is half the Euler Characteristic of its boundary, which is a linear combination of other Euler Characteristics, of classes with more colors. Also the Euler Characteristic of the ambient space~$\chi(M_n)$ is linearly expressed using inclusion-exclusion of color classes, and since $\chi(M_n) = \chi(M) = O(1) = o(n^d)$, this gives an additional linear relation if $d$ is even. In either case, $2^{k-1}-1$ degrees of freedom are left for the expected Euler Characteristic densities, or for the random fluctuations around them. 

For example, if $k=2$ and $d$ is odd then it follows that $\mathcal{E}_B = \mathcal{E}_W = \tfrac12 \mathcal{E}_{BW}$ using boundaries, while if $d$ is even then $\mathcal{E}_{BW} = 0$ and $\mathcal{E}_B + \mathcal{E}_W = 0$ using $\chi(M)$. For $k=3$ and $d$ odd, $\mathcal{E}_{RGB} = 0$, and $\mathcal{E}_{R} = \tfrac12\mathcal{E}_{RG}+\tfrac12\mathcal{E}_{RB}$, and similarly for $\mathcal{E}_{G}$ and $\mathcal{E}_{B}$. If
$k=3$ and $d$ is even then $\mathcal{E}_{RG} = \mathcal{E}_{RB} = \mathcal{E}_{GB} = \tfrac12\mathcal{E}_{RGB}$, and $\mathcal{E}_R + \mathcal{E}_G + \mathcal{E}_B - \mathcal{E}_{RG} - \mathcal{E}_{RB} - \mathcal{E}_{GB} + \mathcal{E}_{RGB} = 0$ using $\chi(M)$. Similarly for $k=4$ colors, there are 8 linear relations between the 15 Euler Characteristics, leaving 7~degrees of freedom.

\begin{example}
\label{1of2colors}
A~single color~$\mathcal{C} = \{C\}$ that occurs with probability $p \in [0,1]$ in a $k$-colored $d$-dimensional manifold, for any $k \geq 2$ and $d \geq 1$, gives a random one-color class of the same dimension. Its expected Euler Characteristic density is 
$$ \mathcal{E}_C(p,\dots) \;=\; \sum_{r=1}^{d+1} \genfrac{\{}{\}}{0pt}{1}{d+1}{r} \, (r-1)! \, (-1)^{r+1} \, p^r $$
The first few expected densities for $d=1,2,3$ are respectively $p-p^2$, $\;p-3p^2+2p^3$, and $p-7p^2+12p^3-6p^4$. The latter changes sign at~$p=\tfrac12(1\pm1/\sqrt{3})$. 

By taking $k=2$ colors $C$~and~$D$, and using the linear relations from the above remark:
$$ \mathcal{E}_C(p,1-p) \;=\; \mathcal{E}_D(1-p,p) \;=\; (-1)^{d+1} \mathcal{E}_C(1-p,p) $$
Hence these polynomials are symmetric around $p=\tfrac12$.
The case $k=2$ was treated in~\cite{bobrowski2020homological} for the torus~$T^d_n$, where the sign changes in $\mathcal{E}_C$ are compared with phase transitions in the topology of the color class.  If $d$ is odd then $\mathcal{E}_C$ is half of $\mathcal{E}_{CD}$, the expected Euler Characteristic density of the two-color class. Letting $d=3$ for example, the above formula reflects how the relative occurrence of 2-spheres, which have positive Euler characteristic, and surfaces of genus greater than one, which have negative Euler characteristic, transitions with~$p \in [0,1]$.
\end{example}

\begin{example}
\label{2of3colors}
Consider a $k$-colored manifold with $k \geq 2$ colors. Every two colors $C,D \in \mathcal{C}$  give a codimension-one submanifold, possibly with boundary. Their expected Euler Characteristic is
$$ \mathcal{E}_{CD}(p,q,\dots) \;=\; \sum_{r=2}^{d+1} \genfrac{\{}{\}}{0pt}{1}{d+1}{r} \, (r-1)! \, (-1)^{r} \left[ (p+q)^r - p^r -q^r \right] $$
For example in dimension $d=3$, one obtains
$14 p q - 36 p^2 q - 36 p q^2 + 24 p q^3 + 36 p^2 q^2 + 24 p^3 q$, which vanishes on the ellipse $21 (p - q)^2 + 3 (7 p + 7 q - 6)^2 = 10$, being negative inside.
\end{example}

\begin{example}
\label{surfaces}
The all-color class in a random $(d-1)$-coloring of a $d$-dimensional manifold is a closed surface. Its expected Euler Characteristic density is
\begin{align*}
\mathcal{E}_{\mathcal{C}}(p_1,\dots,p_{d-1}) \;=\;& \left[\tbinom{d+1}{3} + 3\tbinom{d+1}{4} \right] \,(d-2)!\, \left[ (d-1)! \prod\nolimits_ip_i \right] \\
&\;-\; \left[\tbinom{d+1}{2}\right] \,(d-1)!\, \left[ \tfrac{d!}{2!} \sum\nolimits_jp_j \prod\nolimits_ip_i \right] \\
&\;+\; \left[ \,1 \left] \,d!\, \left[ \tfrac{(d+1)!}{3!} \sum\nolimits_jp_j^2 \prod\nolimits_ip_i \;+\; \tfrac{(d+1)!}{2!2!} \sum\nolimits_{j<l}p_jp_l \prod\nolimits_ip_i \right] \right.\right. \end{align*} 
The factors in the left brackets are the Stirling numbers $\genfrac{\{}{\}}{0pt}{1}{d+1}{d-1}$, $\genfrac{\{}{\}}{0pt}{1}{d+1}{d}$, $\genfrac{\{}{\}}{0pt}{1}{d+1}{d+1}$, and those in the right brackets are respective probabilities of all $d-1$ colors to occur in a simplex. Both of these quantities are computed here directly rather than by inclusion-exclusion as in the statement of the formula. The proof below shows where they come from. This simplifies to
$$ \mathcal{E}_{\mathcal{C}}(p_1,\dots,p_{d-1}) \;=\; \frac{(d+1)!\,(d-1)!}{24}\,\left(d \sum\limits_{j=1}^{d-1} p_j^2 \;-\; 2 \right) \prod\limits_{i=1}^{d-1}p_i $$

By this formula, the expected Euler Characteristic density is negative within a ball of radius $\sqrt{1-(2/d)} \, / (d-1)$ in the parameter space, centered at the middle point where each color occurs with probability $p_i = \tfrac{1}{d-1}$. For example, in a 3-colored 4-dimensional manifold with probability vector $p = (r,g,b)$, the 3-color class is a surface with expected Euler Characteristic density $60rgb(2r^2+2g^2+2b^2-1)$, and thus vanishes on a circle of radius $1/\sqrt{6}$ around $(\tfrac13,\tfrac13,\tfrac13)$ in the $r+g+b=1$ plane.
\end{example}

\begin{proof}[Proof of Theorem~\ref{ec}]
Let $M$ be a closed triangulated $d$-manifold and let $\mathcal{C} \subseteq \{C_1,\dots,C_k\}$ be a nonempty subset of the colors, as in the statement of the theorem. Also, let $M_n$ be the sequence of refined triangulations of~$M$ as in Definition~\ref{scheme}, and let $N_n \subseteq M_n$ the respective color classes that correspond to~$\mathcal{C}$, in a random $k$-coloring of the vertices, independently according to the probability vector $p = (p_1,\dots,p_k)$, as in Definition~\ref{eec}. 

Fix a coloring of the vertices of~$M_n$, and denote by $V_{\sigma}$ the set of vertices of a simplex $\sigma$ in~$M_n$, so that $|V_{\sigma}| = \dim \sigma+1$. By Lemma~\ref{manifold_in_simplex} and the proof of Lemma~\ref{submanifold}, the color class~$N_n$ corresponding to~$\mathcal C$ admits a cell complex structure, where for every $\sigma$ in~$M_n$ one of the following three cases holds. If some color in~$\mathcal{C}$ is missing from the colors assigned to~$V_{\sigma}$, then $\sigma$~is disjoint from~$N_n$. If the colors that appear in~$V_\sigma$ are all the colors in~$\mathcal{C}$ and no others, then $N_n \cap \sigma$ is a disk of dimension~$|V_{\sigma}| - |\mathcal{C}|$ whose boundary is~$N_n \cap \partial\sigma$. In the remaining case that $V_\sigma$~are assigned all the colors in~$\mathcal{C}$ plus some other colors, $N_n \cap \sigma$ is a disk~$D_{\sigma}$ of dimension~$|V_{\sigma}| - |\mathcal{C}|$. The boundary of the disk~$D_{\sigma}$ intersects the interior of~$\sigma$ as an open disk of dimension~$|V_{\sigma}| - |\mathcal{C}|-1$.

We compute the Euler Characteristic of $N_n$ from this cell decomposition, as an alternating count of the cells of all dimensions: $\chi(N_n) = \sum_i(-1)^i f_i(N_n)$ where $f_i(N_n)$ is the number of $i$-dimensional cells. The contribution of the third case listed above to the Euler Characteristic cancels due to the alternating signs, because the interior of such a simplex~$\sigma$ contains two open cells of consecutive dimensions. It is hence sufficient to count only simplices assigned \emph{exactly} the colors in~$\mathcal{C}$. 

Therefore, in every term of the above sum over~$i \geq 0$, we replace~$f_i(N_n)$ with the number of $\mathcal{C}$-colored simplices of~$M_n$ whose $r=|\mathcal{C}|+i$ vertices are colored exactly by~$\mathcal{C}$. Initially, there are $f_{r-1}(M_n)$ many $r$-vertex simplices in $M_n$, and then each one is exactly $\mathcal{C}$-colored with the same probability~$P_r(\mathcal{C})$, to be computed below. By linearity of expectations, the expected density of $\chi(N_n)$ becomes
$$ \mathcal{E}_{\mathcal{C}}(p_1,\dots,p_k) \;=\; \frac{1}{f_0(M_n)} \sum\nolimits_{r=|\mathcal{C}|}^{d+1} \;(-1)^{r-|\mathcal{C}|}\; f_{r-1}(M_n)\; P_r(\mathcal{C}) $$

In order to estimate $f_{r-1}(M_n)$, we first count the $r$-vertex simplices in the flat torus $M_n=T^d_n$ and later generalize to other closed triangulated $d$-manifolds $M_n$. The number of vertices is $F_0(T^d_n) = n^d$. For each vertex $v \in \{0,1,2,\dots,n-1\}^d$, we can separately count the $r$-vertex simplices that are incident to $v$ in the small cube $v+[0,1]^d$. Indeed, in the braid triangulation used in $T_n^d$, the vertices of every $r$-vertex simplex are $v=v_1 < v_2 < \dots < v_r$, increasing in the partial order of~$\{0,\dots,n\}^d$. Their differences $v_2-v_1$, $v_3-v_2, \dots$ are nonzero in $\{0,1\}^d$ with disjoint supports, so that also $v_r-v \in \{0,1\}^d$. Thus, an $r$-vertex simplex corresponds to a partition of some subset of $\{1,\dots,d\}$ into $r-1$ nonempty labeled parts. The number of partitions of $a$ items into $b$ unlabeled nonempty parts is the Stirling number of the second kind $\genfrac{\{}{\}}{0pt}{}{a}{b}$, so their labelled count is $\genfrac{\{}{\}}{0pt}{}{a}{b}b!$. We divide into cases according to whether the total difference is $v_r-v=(1,\dots,1)$ which splits into $r-1$ differences, or there is an $r$th part of unused coordinates. The total count is therefore 
$$ f_{r-1}(T^d_n) \;=\; n^d \left(\textstyle \genfrac{\{}{\}}{0pt}{}{d}{r-1}\,(r-1)! + \genfrac{\{}{\}}{0pt}{}{d}{r}\,r! \right) \;=\; n^d\, \textstyle \genfrac{\{}{\}}{0pt}{}{d+1}{r} \,(r-1)! $$ 
using a standard combinatorial recurrence. See the statement of the theorem for a standard formula for computing the Stirling numbers. We note that that formula counts partitions without empty parts based on inclusion-exclusion of partitions with empty parts.

The probability that $r$ given vertices $V_\sigma$ are assigned colors from the set~$\mathcal{C}$ is $(\sum_{C_i \in \mathcal{C}}p_i)^r$. However, this event includes colorings that only use a proper subset of~$\mathcal{C}$, and those should be excluded from $P_r(\mathcal{C})$. For example, $P_r(\{C_i,C_j\}) = (p_i+p_j)^r-p_i^r-p_j^r$ since we need both colors to occur in~$V_\sigma$. More generally, by inclusion-exclusion over all subsets $X \subseteq \mathcal{C}$, the probability that the colors that occur in~$V_{\sigma}$ are exactly $\mathcal{C}$ is
$$ P_r(\mathcal{C}) \;=\; \sum_{X \subseteq \mathcal{C}} (-1)^{|\mathcal{C} \setminus X|} \left(\sum\nolimits_{C_i \in X} p_i \right)^r $$
Substituting $f_{r-1}(T^d_n)$ and $P_r(\mathcal{C})$ in the above sum for $\mathcal{E}_\mathcal{C}(p_1,\dots,p_k)$ gives the exact formula stated in the theorem for every $n$, and hence in the limit $n \to \infty$ as well.

Finally, we generalize from $T^d_n$ to $M_n$. We recall Definition~\ref{scheme} of the subdivisions~$M_n$. A~triangulated $d$-manifold $M$ with $F_d(M)$ top simplices, is first subdivided into $(d+1)F_d(M)$ large cubes. Each one of these cubes is further subdivided into $n^d$ small cubes, which are subdivided again using the braid triangulation. Note that each of the $(d+1)F_d(M)$ large cubes is structured exactly as the flat torus~$T^d_n$, except for the $d(d+1)F_d(M)$ cubical boundaries between the large cubes, which altogether contain at most $d(d+1)F_d(M)n^{d-1}$ vertices. Since the construction of the subdivision is such that the valency of vertices is bounded as $n$ grows, only $O(n^{d-1})$ faces of any dimension are incident to such vertices, and the other $O(n^d)$ lie within the large cubes which are triangulated similar to $T^d_n$. In conclusion, the number of vertices of $M_n$ is $(d+1)F_d(M)n^d \pm O(n^{d-1})$, and similarly for higher-dimensional faces $f_{r-1}(M_n) = (d+1)F_d(M)\,f_{r-1}(T^d_n) \pm O(n^{d-1})$ for each $r \leq d+1$. It follows that the ratio giving the Euler Characteristic density for $M_n$ is the same as for the flat torus $T^d_n$ up to a negligible $(1 \pm O(1/n))$ factor.
\end{proof}

\subsection{Percolation}
\label{percolation}
Consider a triangulated manifold~$M$, with a large or infinite number of vertices, and color them as above, independently at random using a~probability vector $p = (p_1,\dots,p_k)$. The color class of all $k$ colors, and other color classes, may have many connected components. The basic phenomenon of percolation is that random subsets of a space, such as these color classes, have different phases of connectivity.  Typically, the connected components are small and bounded for some values of the parameters~$p$, while giant or infinite components emerge for other values, and the transition is usually sharp.

The percolation properties of random color classes can be studied under various scenarios for the topology of the ambient space~$M$. In one standard setting, $M$~is a triangulation of~$\mathbb{R}^d$ with vertices~$\mathbb{Z}^d$, as in Definition~\ref{standard}. For $k = 2$ colors, the behavior of the color classes are well-studied cases in percolation theory. Sheffield and Yadin~\cite{sheffield2014tricolor} initiated the mathematical study of 3-color percolation in three dimensions, which had only been considered previously in the physics literature~\cite{bradley1992growing, nahum2012universal}. By a clever topological argument, they established the existence of some $p=(p_1,p_2,p_3)$ that gives rise to long 3-color components with high probability. They worked with a tessellation dual to ours, which actually leads to an equivalent model.  

A~key advantage of studying percolation in these coloring-based models is that the random components are well-behaved submanifolds, of a preferred dimension. This is usually not the case in models of percolation, where typical random subsets have many singularities. Since the all-color classes are proper submanifolds, powerful tools such as cohomology and intersection theory may be applied, as done in~\cite{sheffield2014tricolor}.

The event of percolation may be precisely defined by several different notions of large-scale connectivity, which are sometimes known or believed to occur in the same phases in the parameter space of the random models. One may ask for which parameters~$p$ there exists an infinite component with probability one, or whether the average size of the component containing a given point converges or diverges. Other possible definitions are based on the asymptotics of the probability of ``crossing'' an $n \times \dots \times n$ cube, with components that connect two opposite faces.

Another recent approach is \emph{homological percolation}, introduced by Bobrowski and Skraba \cite{bobrowski2020homological2}, and closely related to Topological Data Analysis. The embedding $f:N \hookrightarrow M$ of a random subset in the given manifold induces maps between the corresponding homology groups $f_{\ast}:H_i(N;R) \to H_i(M;R)$, for every dimension $i \in \{1,\dots,d\}$ and commutative ring of coefficients~$R$. Two events $A_i$ and $E_i$ are defined to occur if, respectively, \emph{any} or \emph{every} nontrivial $i$-cycle in~$H_i(M;R)$ is in the image of $f_\ast$. Intuitively this means that $N$ contains large $i$-cycles at the scale of $M$. The one-dimensional events $A_1$ and $E_1$ are expected to correlate with the common percolation notions, of having unbounded components or ones that ``cross''~$M$ and thus create large cycles. 

This approach is suitable for the study of percolation in closed manifolds $M$ with nontrivial homology. Bobrowski and Skraba \cite{bobrowski2020homological} experiment with various constructions of random subsets in Riemannian manifolds. In our terminology, one of their models is equivalent to the 1-color class~$N$ in a random 2-coloring of the flat $d$-torus $M = T^d_n$, with probabilities $(p,1-p)$. Conveniently, this $M$~has nontrivial homology groups, with $\dim H_i(M;\mathbb{Z}_2) = \tbinom{d}{i}$. Their simulations indicate a dichotomy for the parameter~$p$ between the $i$th \emph{subcritical} phase where $P(A_i) \to 0$ and the \emph{supercritical} phase where $P(E_i) \to 1$. Some of their discoveries have been recently proven by Duncan et~al.~\cite{duncan2021homological}.
The following general question contains settings from \cite{sheffield2014tricolor, bobrowski2020homological,duncan2021homological} as special cases. 

\begin{question}
\label{phases}
For every four positive integers $m \leq k \leq d$ and $i \leq d-m+1$, consider the $n$-subdivided $d$-dimensional torus $M_n=T_n^d$ for $n \in \mathbb{N}$, randomly $k$-colored according to a probability vector  $p=(p_1,\dots,p_k)$, and let $N_n \subset M_n$ be the random $m$-color class corresponding to the colors $\{C_1,\dots,C_m\}$. For which parameters~$p$ does $N_n$ satisfy $P(A_i) \to 0$ as $n \to \infty$? Which $p$ satisfy $P(E_i) \to 1$, if any? Are there other phases?
\end{question}

\begin{remark}
In fact, it is sufficient to consider only $m,k$ such that $m \in \{k-1,k\}$. Indeed, an $m$-color class does not depend on any distinctions between the colors not involved in it, so it only matters if there exist other colors or not. Allowing general $m \in \{1,\dots,k\}$ lets one study the interactions between all the $\scriptscriptstyle{\tbinom{k}{m}}$ different $m$-color classes.
\end{remark}

\begin{remark}
\label{intersection}
In some cases, algebraic considerations provide information on Question~\ref{phases}. For even dimensions~$d$, take $d/2+1$ colors, and consider the $d/2$-homological percolation of the $d/2$-dimensional all-color class~$N$, so that~$m=k=d/2+1$ and $i=d/2$. The connected components of~$N$ generate a subgroup of $H_{d/2}(M; \mathbb{Z})$ under the embedding. If $M$ admits a nontrivial intersection form, then some pair of $d/2$-dimensional submanifolds has a nonzero intersection number, and then the elements represented by these submanifolds cannot both occur in this subgroup. It follows that $P(E_{d/2})=0$ in this setting, regardless of the probabilities $(p_0,\dots,p_{d/2})$ assigned to each color.

For example, a uniformly random $3$-coloring of $M=T^4$ yields a collection of closed surfaces, whose images generate some subgroup of $H_2(M; \mathbb{Z}) \cong \mathbb{Z}^6$. This is a proper subgroup since 2-cycles that represent an $x_1x_2$-torus and an $x_3x_4$-torus must intersect. When such topological constraints exist, it is natural to ask about an adjusted event $E_i'$ that a subgroup of maximal rank is generated by the color class. \end{remark}

Suppose that there exist two phases of probability vectors $(p_1,\dots,p_k)$ for some $d,k,m,i$ as in Question~\ref{phases} above: a~subcritical phase where $P(A_i) \to 0$ and a supercritical phase where $P(A_i) \to 1$. Here are some natural questions about their universality. This list includes and extends questions from~\cite{bobrowski2020homological2}.

\begin{samepage}  \begin{itemize}
\itemsep0.125em
\item 
Do these phases extend across all manifolds~$M$, using the sequence of subdivisions $M_n$ constructed in Definition~\ref{scheme}?
\item 
In the supercritical phase, does $P(E_i') \to 1$ as well? Note that we use $A_i$ and~$E_i'$ since $E_i$ is impossible in some cases as in Remark~\ref{intersection}.
\item
Considering different rings of coefficients~$R$ such that $H_i(M;R)$ is nontrivial, do the two phases depend on~$R$? 
\item 
If the fundamental group of~$M$ is nonabelian, do the phases for $A_1$ correspond to the emergence of loops in~$N$ that are nontrivial in~$\pi_1(M)$? that generate~$\pi_1(M)$?
\item 
Does the supercritical phase for $A_1$ in this model coincide with other events of percolation, such as the emergence of large or infinite components? 
\end{itemize}  
\end{samepage}

\subsection{Discussion}

The Euler Characteristic $\chi(N)$, studied in Section~\ref{eec-section}, is only one topological property of the random submanifolds $N$ that arise as color classes in this model. Among other properties that are natural to study are the Betti numbers $b_i(N) = \mathrm{rank}\, H_i(N)$. These numbers refine the Euler Characteristic via $\chi(N) = \sum_{i\geq 0}(-1)^i b_i(N)$.

The typical behavior of the Betti numbers, often observed in simulations of random models, and in various constructions from topological data analysis, is that one of the~$b_i(N)$ is expected to be more dominant than the others, see~\cite[for example]{edelsbrunner2008persistent,adler2010persistent,kahle2013limit}. In such circumstances, the expected Euler Characteristic may be viewed as an accessible approximation to the magnitude of that Betti number. These random constructions are usually parametrized, and the dominant dimension~$i$ may vary with the parameters. Then, the zeros of the expected Euler Characteristic may approximate the boundaries between regions in the parameter space of different dominant homology groups. 

A~related curious phenomena of this type has been discovered by Bobrowski and Skraba \cite{bobrowski2020homological}. They have found experimentally that the zeros of the expected Euler Characteristic are close to the phase transitions of \emph{homological percolation}. One of the random models they have investigated is the 1-color class of a random 2-coloring, as discussed in Example~\ref{1of2colors} above. This naturally raises the question whether these phenomena extend to general $m$-color submanifolds in random $k$-colorings. 

Take the $3$-colored 4-dimensional torus $T^4_n$ for example. Its parameter space is the triangle $\{(r,g,b) : r+g+b=1\}$. We have seen in Example~\ref{surfaces} that the expected Euler Characteristic of the 3-color class is positive outside the circle $2r^2+2g^2+2b^2=1$ and negative inside. The Euler Characteristic tells us that $b_1 > 2b_0$ on average inside the circle. On the other hand, we expect $2$-dimensional homological percolation in some region within the triangle, namely, having a nontrivial surface in the second homology group of the 4-torus. Is the circle defined by the Euler Characteristic related to the domain where percolation occurs?

In the case of the 1-dimensional 3-color class in a 3-colored 3-torus, this 1-manifold does not have a variable Euler Characteristic to compare to its percolation domain. However, one may try to relate it to the three Euler Characteristics of the 2-color classes, shown in Example~\ref{2of3colors} above to vanish on certain ellipses in the triangular parameter space. Is the percolation domain from Figure~9 in~\cite{sheffield2014tricolor} approximated by some condition on these three numbers? In this regard, we remark that the threshold 0.2459615 quoted in this figure from~\cite[b.c.c.~lattice]{lorenz1998universality} refers to the wrong lattice, and should be $0.169 \pm 0.002$ as in~\cite[b.c.c.~NN+2NN lattice]{domb1966crystal,dalton1966crystal, jerauld1984percolation}.

\section{Four-Ball Genus}
\label{genus4}
In this section we discuss preliminary experimental work applying our model to search for topologically interesting surfaces in 4-manifolds.
A fundamental issue in the study of smooth 4-dimensional manifolds is the question of when a curve in the boundary of a 4-manifold bounds a smooth disk. Finding embedded disks is a crucial step in performing smooth surgery on 4-manifolds. An important special case is the question of whether a given knot in the $3$-sphere is smoothly \emph{slice}, meaning that it bounds a smooth disk embedded in the 4-ball. There is no known algorithm that determines whether or not a knot is slice. The 4-ball genus computation is challenging even for small knots. The KnotInfo database lists invariants of knots with up to 12 crossings~\cite{knotinfo}.  The task of identifying the slice knots and computing the 4-ball genus of the others has only recently been completed~\cite{brittenham2021smooth}. Most of the recent progress is due to getting better upper bounds for the 4-ball genus by explicit construction of a low-genus surface in the 4-ball. This has been done by various methods suitable for different situations, and often by reducing to a known case of another knot. It is desirable to develop general methods to effectively find low-genus surfaces spanning a given knot. The relevance of this question to central questions of 4-manifold topology is reflected in a recent paper of Manolescu and Piccirillo, in which they produce five knots that bound topological disks in the 4-ball \cite{manolescu2021zero}. They show that if any of these knots bound smooth disks then there exists an exotic smooth 4-sphere, a counterexample to the four-dimensional smooth Poincar\'e conjecture. The 4-ball genus problem in general  presents a challenge to our model. If we fix a knot~$K$, we can compute many random surfaces in $B^4$ with boundary~$K$, and look among them for a low genus surface.  If $K$ is smoothly slice then such a process will find a disk with nonzero probability, while otherwise it is doomed~to~fail. 

We indicate one approach to such a process. We begin by coloring a triangulation of $S^3 = \partial B^4$ with three colors so that the 3-color class is precisely the knot~$K$. We then randomly color an extension of this triangulation to $B^4$. This produces a 3-color class that consists of a connected orientable surface~$F$ whose boundary is precisely~$K$, along with closed components not meeting~$K$. It is straightforward to compute the genus of~$F$ and in this way to get an upper bound for the 4-ball genus of~$K$. The main effort is to find a coloring of~$B^4$ that minimizes this bound, and for this problem we design a local search algorithm. Initially, the 3-color surface has many other components and $F$ has large genus. We iteratively recolor internal vertices of~$B^4$, and check the effect on the surface. We bias towards accepting changes that decrease the genus of~$F$, or that of the whole 3-color class, and rejecting those that increase genus. We decide whether to recolor by tossing a biased coin, and follow the heuristic of simulated annealing~\cite[\S4]{aarts2003local}. In this section, we discuss the theoretical justification and practical implementation of these experiments.

We now give some details of how we set up the triangulated 4-ball and the boundary conditions on the 3-sphere. Recall Definition~\ref{standard} above, which realizes the 3-sphere in 4-space by restricting the triangulated grid with vertices $\mathbb{Z}^4$ to the boundary of an $n \times n \times n \times n$ box. The following proposition, starting with a given knot or link~$L$, constructs a 3-coloring of the 3-sphere so that $L$ is realized as the entire 3-color class. Unlike Proposition~\ref{link} above, this construction is explicit and given in terms of simple building blocks called \emph{tiles}, that implement various local features of a link diagram. 

\begin{proposition}
\label{tiling}
Every link $L$ can be effectively constructed as the 3-color class in a 3-colored triangulated 3-sphere~$S_n^3 = \partial B_n^4$, for a sufficiently large 4-ball $B_n^4$.
\end{proposition}

\begin{proof}
We start from a planar diagram of~$L$. We need a diagram whose complement has at most two faces with an odd number of sides. The existence of such diagrams was shown by Adams, Shinjo, and Tanaka, using \emph{meander diagrams} for knots~\cite{adams2011complementary}. These results were extended by Even-Zohar, Hass, Linial and Nowik to links with an arbitrary number of components, using \emph{potholder diagrams}~\cite{even2019universal}. The proofs in these two works are constructive, providing algorithms that compute a new diagram of a given link $L$, with at most two odd-sided faces. We note however that these constructions often introduce new crossings, and it may be desirable to keep the diagrams small.

We isotope the diagram of $L$ in the 2-sphere, and draw it on the boundary of an $n_x \times n_y \times n_z$ box. This can be done so that the resulting diagram is made of the following 2-dimensional tiles and their reflections and rotations.
$$
\tikz{
\path[draw=black, line width=0.5, fill=black!10] (0,0) -- (0,1) -- (1,1) -- (1,0) -- cycle;
\draw[line width=4, black] (0.5,0) -- (0.5,1);
\clip (-0.25,-0.25) rectangle (1.25,1.25);
}\;\;\;\;
\tikz{
\path[draw=black, line width=0.5, fill=black!10] (0,0) -- (0,1) -- (1,1) -- (1,0) -- cycle;
\draw[line width=4, black] (0.5,0) to[in=0,out=90] (0,0.5);
\clip (-0.25,-0.25) rectangle (1.25,1.25);
}\;\;\;\;
\tikz{
\path[draw=black, line width=0.5, fill=black!10] (0,0) -- (0,1) -- (1,1) -- (1,0) -- cycle;
\draw[line width=4, black] (0.5,0) -- (0.5,1);
\draw[line width=4, black] (0,0.5) -- (0.35,0.5);
\draw[line width=4, black] (1,0.5) -- (0.65,0.5);
\clip (-0.25,-0.25) rectangle (1.25,1.25);
}\;\;\;\;
\tikz{
\path[draw=black, line width=0.5, fill=black!10] (0,0) -- (0,1) -- (1,1) -- (1,0) -- cycle;
\clip (-0.25,-0.25) rectangle (1.25,1.25);
}
$$
It is clearly sufficient to draw $L$ on one of the six faces of the box and put empty tiles on the others. However, it may be useful to use all six faces, to keep the dimensions of the box small. See Figure~\ref{tilesdiagram} for an example of a tiles diagram of a link on the boundary of a box. 

\begin{figure}
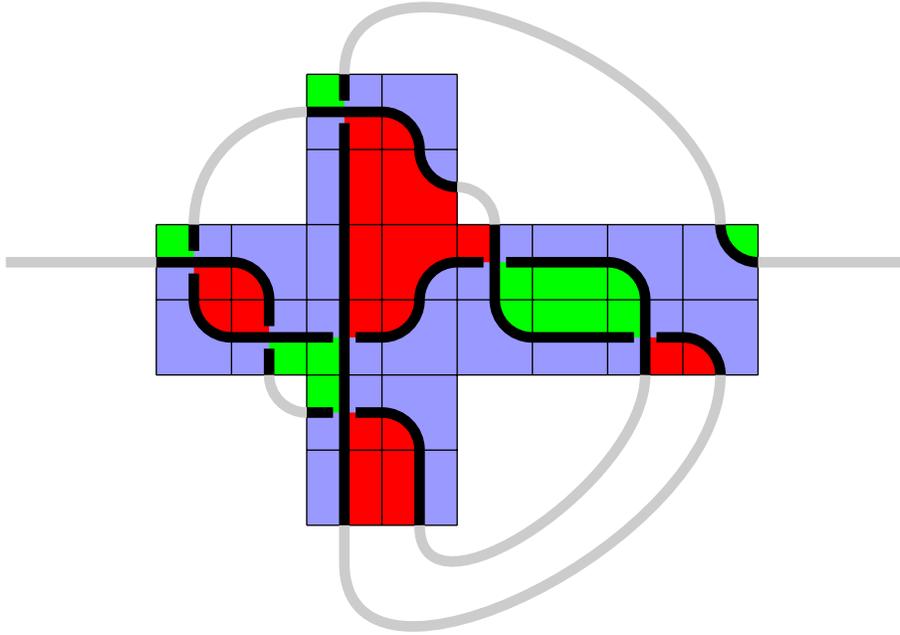

\centering
\vspace{-2em}
\tikz{
\fill[blue!40] (0,2) -- (2,2) -- (2,0) -- (4,0) -- (4,2) -- (8,2) -- (8,4) -- (4,4) -- (4,6) -- (2,6) -- (2,4) -- (0,4) -- cycle;
\fill[green] (2,6) -- (2.5,6) -- (2.5,5.5) -- (2,5.5) -- cycle
(0,4) -- (0,3.5) -- (0.5,3.5) -- (0.5,4) -- cycle
(8,4) -- (7.5,4) to[out=270,in=180] (8,3.5) -- cycle
(4.5,3.5) -- (6,3.5) to[out=0,in=90] (6.5,3) -- (6.5,2.5) -- (5,2.5) to[out=180,in=270] (4.5,3) -- cycle
(1.5,2.5) -- (2.5,2.5) -- (2.5,1.5) -- (2,1.5) -- (2,2) -- (1.5,2) -- cycle;
\fill[red] (4,4.5) to[out=180,in=270] (3.5,5) to[out=90,in=0] (3,5.5) -- (2.5,5.5) -- (2.5,2.5) -- (3,2.5) to[out=0,in=270] (3.5,3) to[out=90,in=180] (4,3.5) -- (4.5,3.5) -- (4.5,4) -- (4,4) -- cycle
(2.5,1.5) -- (2.5,0) -- (3.5,0) -- (3.5,1) to[out=90,in=0] (3,1.5) -- cycle
(6.5,2) -- (6.5,2.5) -- (7,2.5) to[out=0,in=90] (7.5,2) -- cycle
(1.5,2.5) -- (1.5,3) to[out=90,in=0] (1,3.5) -- (0.5,3.5) -- (0.5,3) to[out=270,in=180] (1,2.5) -- cycle;
\path[draw=black, line width=0.5] 
(2,0) -- (4,0) 
(2,1) -- (4,1) 
(0,2) -- (8,2) 
(0,3) -- (8,3) 
(0,4) -- (8,4) 
(2,5) -- (4,5) 
(2,6) -- (4,6) 
(0,2) -- (0,4)
(1,2) -- (1,4)
(2,0) -- (2,6)
(3,0) -- (3,6)
(4,0) -- (4,6)
(5,2) -- (5,4)
(6,2) -- (6,4)
(7,2) -- (7,4)
(8,2) -- (8,4);
\draw[line width=4, black] 
(4,4.5) to[out=180,in=270] (3.5,5) to[out=90,in=0] (3,5.5) -- (2,5.5)
(0.5,4) -- (0.5,3.65)
(0.5,3.35) -- (0.5,3) to[out=270,in=180] (1,2.5) -- (2.35,2.5)
(2.65,2.5) -- (3,2.5) to[out=0,in=270] (3.5,3) to[out=90,in=180] (4,3.5) -- (4.35,3.5)
(4.65,3.5) -- (6,3.5) to[out=0,in=90] (6.5,3) -- (6.5,2)
(3.5,0) -- (3.5,1) to[out=90,in=0] (3,1.5) -- (2.65,1.5)
(2.35,1.5) -- (2,1.5)
(1.5,2) -- (1.5,2.35)
(1.5,2.65) -- (1.5,3) to[out=90,in=0] (1,3.5) -- (0,3.5)
(8,3.5) to[out=180,in=270] (7.5,4)
(2.5,6) -- (2.5,5.65)
(2.5,5.35) -- (2.5,0)
(7.5,2) to[out=90,in=0] (7,2.5) -- (6.65,2.5)
(6.35,2.5) -- (5,2.5) to[out=180,in=270] (4.5,3) -- (4.5,4);
\draw[line width=4, black!20] 
(4.5,4) to[out=90,in=0] (4,4.5)
(2,5.5) to[out=180,in=90] (0.5,4)
(6.5,2) to[out=270,in=270] (3.5,0)
(2,1.5) to[out=180,in=270] (1.5,2)
(0,3.5) -- (-2,3.5)
(10,3.5) -- (8,3.5)
(7.5,4) to[out=90,in=90] (2.5,6)
(2.5,0) -- (2.5,-0.5) to[out=270,in=270] (7.5,2);
\clip (-2,0) rectangle (10,6);
}
\vspace{-2em}
\caption{A 7-crossing diagram of the knot $6_1$ made of tiles on the unfolded boundary of a $2 \times 2 \times 2$ box. The faces are colored blue and red/green according to the checkerboard rule. Using such a diagram, we construct a 3-sphere with 3 colors that meet along the link.}
\label{tilesdiagram}
\end{figure}

We then color blue some of the faces of the link diagram according to the \emph{checkerboard coloring}. This means that every segment of the curve that represents the link has a blue face on one side and a white face of the other. By parity considerations, this can be done so that both odd-sided faces, if they exist, are \emph{not} colored blue. Then, we color the remaining faces red and green, so that every crossing of the link diagram is adjacent to one green face and one red face, in addition to the two blue ones. The condition that all blue faces are even-sided guarantees that this can be done consistently. See Figure~\ref{tilesdiagram} for an example of such a face-colored tiles diagram.

Based on the diagram of~$L$ using square tiles on the boundary of an $n_x \times n_y \times n_z$ box, we consider the triangulated 4-ball $B=[0,3n_x] \times [0,3n_y] \times [0,3n_z] \times [0,3]$ and realize $L$ in the triangulated 3-sphere $\partial B$. The vertices of $\partial B$ are colored according to the tiles using the substitution rule of Figure~\ref{tiles}. This rule determines the colors of the $4^3$ vertices of each $3 \times 3 \times 3$ configuration, according to the corresponding square tile, and depending also on the colors of faces. Each $3 \times 3 \times 3$ tile is naturally oriented such that vertices lying on the two hyperplanes $(x,y,z,0)$ and $(x,y,z,3)$ are all colored blue, in accordance with the bottom and top layers respectively in Figure~\ref{tiles}. One can verify that adjacent tiles always agree on their common vertices, so that the vertex coloring of $\partial B$ is well-defined.

We claim that the resulting 3-colored 3-sphere realizes the given link $L$ as its 3-color class. Note that boundary of the 4-dimensional box $B$ is composed of eight 3-dimensional boxes. Six of these faces make a thickened spherical shell $\partial([0,3n_x] \times [0,3n_y] \times [0,3n_z]) \times [0,3]$. The other two 3-boxes in $\partial B$, which are the inside and outside of this shell, are colored blue only, and do not intersect the 3-color class. Projection along the last component takes this 3-dimensional shell to the 2-sphere $\partial([0,n_x] \times [0,n_y] \times [0,n_z])$, up to dilation by~3. The $3 \times 3 \times 3$ tiles are designed such that this projection takes the 3-color class back to the original square tile in the diagram of~$L$. This is easy to verify by examining all small $1 \times 1 \times 1$ cubes that have vertices of all three colors in Figure~\ref{tiles}.

\begin{figure}
\begin{center}
\begin{tabular}{ccccc}
\\
\tikz{
\path[draw=black, line width=0.5, fill=blue!40] (0,0) -- (0,1) -- (1,1) -- (1,0) -- cycle;
\clip (-0.25,-0.25) rectangle (1.25,1.25);
} &
\tikz{
\path[draw=black, line width=0.5, fill=green] (0,0) -- (0,1) -- (1,1) -- (1,0) -- cycle;
\clip (-0.25,-0.25) rectangle (1.25,1.25);
} &
\tikz{
\fill[blue!40] (0,0) -- (0,1) -- (0.5,1) -- (0.5,0) -- cycle;
\fill[red] (1,0) -- (1,1) -- (0.5,1) -- (0.5,0) -- cycle;
\path[draw=black, line width=0.5] (0,0) -- (0,1) -- (1,1) -- (1,0) -- cycle;
\draw[line width=4, black] (0.5,0) -- (0.5,1);
\clip (-0.25,-0.25) rectangle (1.25,1.25);
} &
\tikz{
\fill[blue!40] (0,0) -- (0,1) -- (0.5,1) -- (0.5,0) -- cycle;
\fill[green] (1,0) -- (1,1) -- (0,1) -- (0,0.5) to[out=0,in=90] (0.5,0) -- cycle;
\path[draw=black, line width=0.5] (0,0) -- (0,1) -- (1,1) -- (1,0) -- cycle;
\draw[line width=4, black] (0.5,0) to[in=0,out=90] (0,0.5);
\clip (-0.25,-0.25) rectangle (1.25,1.25);
} &
\tikz{
\fill[blue!40] (0,0) -- (0,1) -- (1,1) -- (1,0) -- cycle;
\fill[green] (1,0.5) -- (1,0) -- (0.5,0) -- (0.5,0.5) -- cycle;
\fill[red] (0.5,1) -- (0,1) -- (0,0.5) -- (0.5,0.5) -- cycle;
\path[draw=black, line width=0.5] (0,0) -- (0,1) -- (1,1) -- (1,0) -- cycle;
\draw[line width=4, black] (0.5,0) -- (0.5,1);
\draw[line width=4, black] (0,0.5) -- (0.35,0.5);
\draw[line width=4, black] (1,0.5) -- (0.65,0.5);
\clip (-0.25,-0.25) rectangle (1.25,1.25);
}
\\[0.5em]
\includegraphics[scale=0.333]{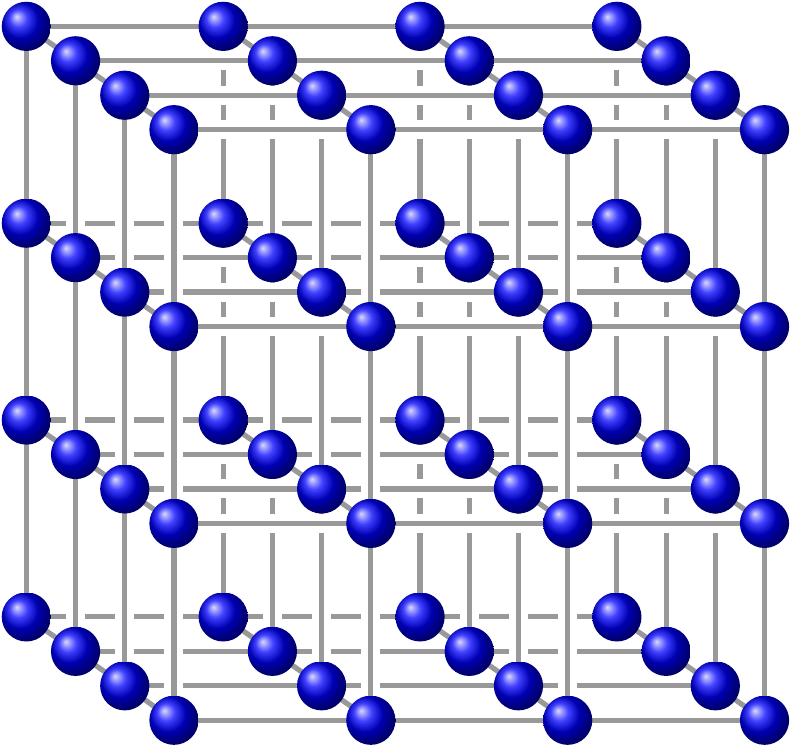} &
\includegraphics[scale=0.333]{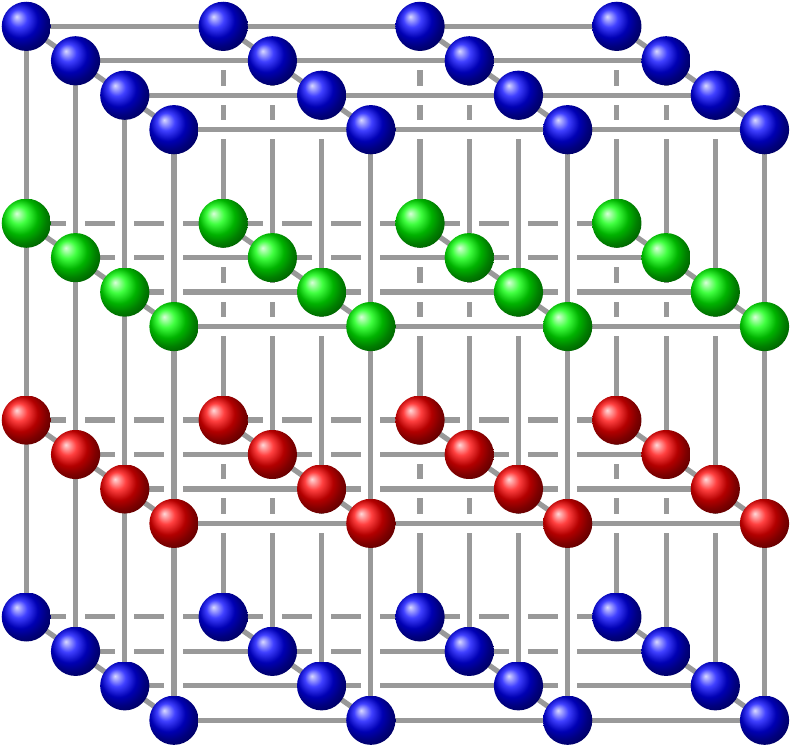} &
\includegraphics[scale=0.333]{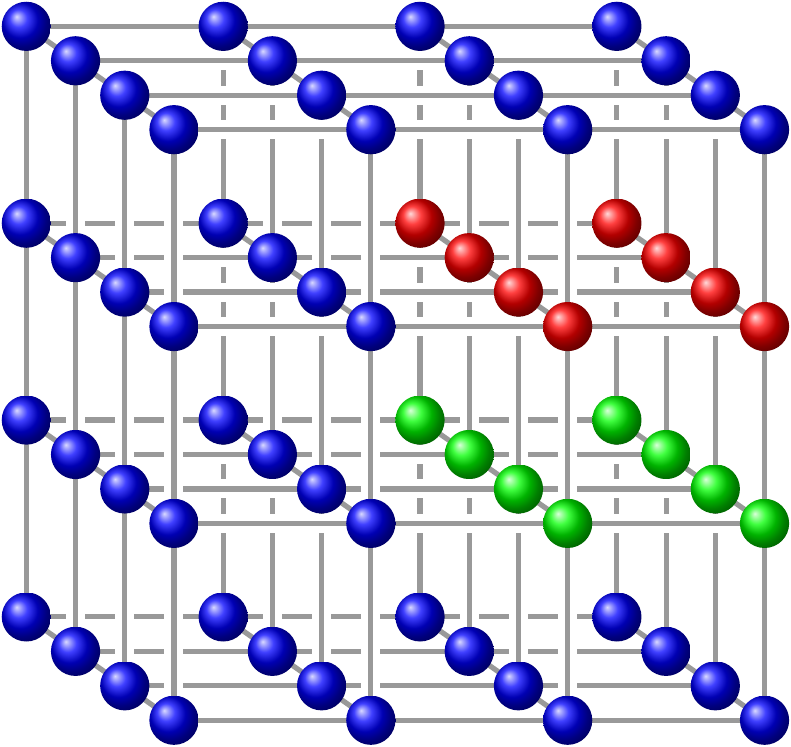} &
\includegraphics[scale=0.333]{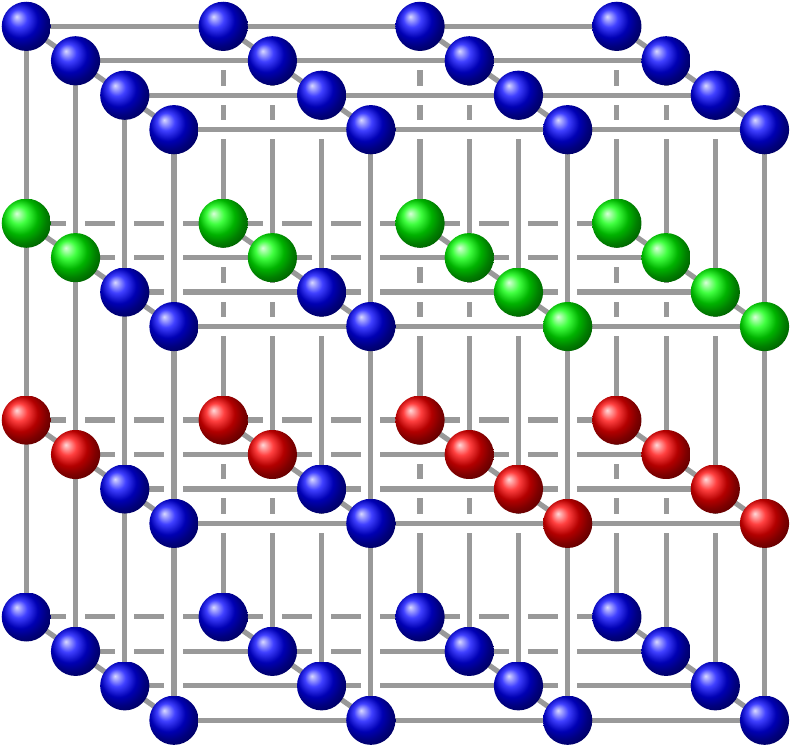} &
\includegraphics[scale=0.333]{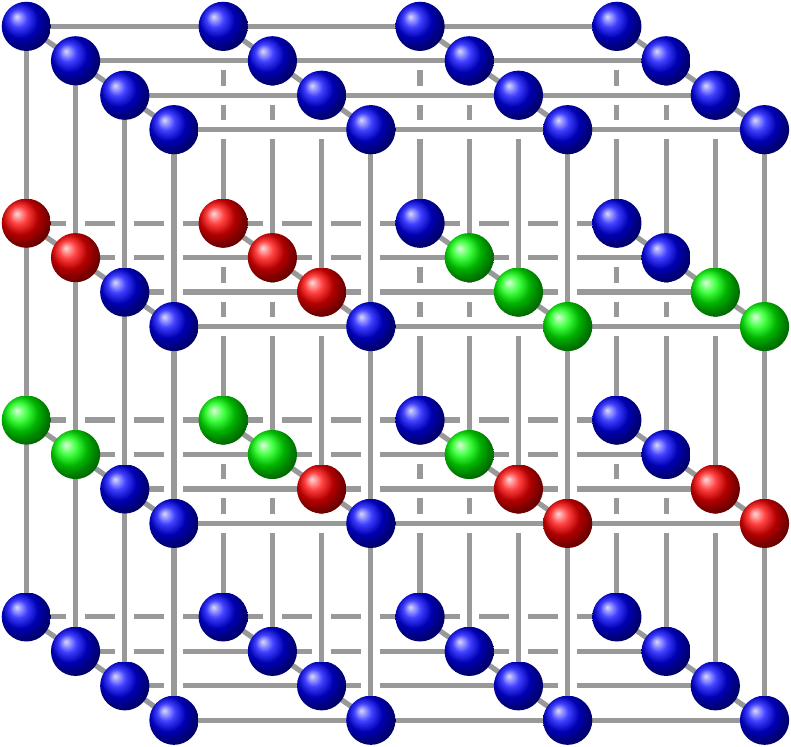} \\[0.5em]
\end{tabular}
\end{center}
\caption{The 3-coloring of each $3 \times 3 \times 3$ tile that corresponds to a 2-dimensional square tile. Obvious adjustments apply for rotations, reflections, and switching red and green.}
\label{tiles}
\end{figure}

In conclusion, the 3-color class of $\partial B$ is isotopic to the given link~$L$. We remark that this holds for our $3 \times 3 \times 3$ tiles regardless of how the four axes are oriented. The triangulation does depend on the axes' directions, as they determine which diagonal faces appear in each small cube. However, the small squares through which the link passes from one cube to another are exactly those where all 3 colors appear, for any choice of diagonals.

In order to transform the 3-sphere from $\partial B$ to $\partial B_n^4 = S_n^3$ as stated, one may let $n=\max(3n_x,3n_y,3n_z)$ and duplicate axis-parallel layers as needed. For example, $n_x$ may be increased to $n_x+1$ so that the 4-dimensional box becomes one longer in the $x$ direction, and the coloring of some layer $x \in \{0,\dots,n_x\}$ is duplicated into two layers. The 3-color class remains isotopic to~$L$. Indeed, every layer intersects the link transversely in the middle of each 3-colored triangle. Duplicating a layer replaces it by a triangular prism where the link crosses from one base to the opposite one, and no other 3-color components are created.
\end{proof}

\begin{remark}
\label{limitlink}
Consider a 3-colored 3-sphere $S_n^3$ realizing a link~$L$ as in Proposition~\ref{tiling}. The \emph{$m$-refinement} is the 3-colored 3-sphere $S_{nm}^3$ where every layer of the 3-coloring is duplicated $m$ times. By the same argument as in the last paragraph of the proof, the 3-color class $L_m$ of the \emph{$m$-refinement} realizes~$L$ too. In view of Proposition~\ref{3color4ball}, we expect that for a sufficiently large~$m$, there exists a coloring of the interior vertices of the $m$-refined 4-ball~$B_{mn}^4$ that produces a surface spanning the link~$L_m$ and realizing the 4-ball genus.
\end{remark}

Initial experiments indicate there is potential in this approach. We have implemented the above procedure for inputting a given knot~$K$ and coloring the vertices of a triangulated~$S^3$ such that its 3-color class is isotopic to~$K$. So far we have tried it on various knots with up to 8 crossings. In one experiment, we have taken~$K$ to be the \emph{square knot}, a six-crossing slice knot that has genus 2 in~$S^3$. In approximately  $10^3$ independent runs of our procedure, with approximately $10^5$ recoloring iterations each, we obtained two colorings that produced a genus 1 surface in~$B^4$ spanning $K$. While not yet recognizing a slice knot, this does demonstrate that our method can produce a surface in~$B^4$ with smaller genus than any surface in~$S^3$. Since there is much room to tune and improve our initial method, we regard this as a promising sign for the potential of this approach.

\bibliographystyle{alpha}
\bibliography{main}

\end{document}